\documentclass[11pt]{amsart}%
\usepackage{amssymb}
\usepackage{amsfonts}
\usepackage{amsmath}
\usepackage[perpage]{footmisc}
\usepackage{graphicx}
\usepackage{hyperref}%
\providecommand{\U}[1]{\protect\rule{.1in}{.1in}}
\makeatletter
\@namedef{subjclassname@2010}{  \textup{2010} Mathematics Subject Classification}
\makeatother
\newtheorem{theorem}{Theorem}[section]
\theoremstyle{plain}

\newtheorem{corollary}[theorem]{Corollary}

\newtheorem{definition}{Definition}
\newtheorem{example}[theorem]{Example}

\newtheorem{lemma}[theorem]{Lemma}

\newtheorem{proposition}[theorem]{Proposition}

\numberwithin{equation}{section}
\begin{document}
\title[High density piecewise syndeticity]{High density piecewise syndeticity of sumsets}

\author[Di Nasso et. al.]{Mauro Di Nasso, Isaac Goldbring, Renling Jin, Steven Leth, Martino Lupini,
Karl Mahlburg}
\address{Dipartimento di Matematica, Universita' di Pisa, Largo Bruno
Pontecorvo 5, Pisa 56127, Italy}
\email{dinasso@dm.unipi.it}
\address{Department of Mathematics, Statistics, and Computer Science,
University of Illinois at Chicago, Science and Engineering Offices M/C 249,
851 S. Morgan St., Chicago, IL, 60607-7045}
\email{isaac@math.uic.edu}
\address{Department of Mathematics, College of Charleston, Charleston, SC,
29424}
\email{JinR@cofc.edu}
\address{School of Mathematical Sciences, University of Northern Colorado,
Campus Box 122, 510 20th Street, Greeley, CO 80639}
\email{Steven.Leth@unco.edu}
\address{Department of Mathematics and Statistics, York University, N520
Ross, 4700 Keele Street, M3J 1P3, Toronto, ON, Canada}
\email{mlupini@mathstat.yorku.ca}
\address{Department of Mathematics, Louisiana State University, 228 Lockett
Hall, Baton Rouge, LA 70803}
\email{mahlburg@math.lsu.edu}

\thanks{The authors were supported in part by the American Institute of Mathematics
through its SQuaREs program. I. Goldbring was partially supported by NSF grant
DMS-1262210. M. Lupini was supported by the York University Elia Scholars
Program. K. Mahlburg was supported by NSF Grant DMS-1201435.}
\date{October, 2013}
\dedicatory{ }
\begin{abstract}
Renling Jin proved that if $A$ and $B$ are two subsets of the natural numbers
with positive Banach density, then $A+B$ is piecewise syndetic. In this paper,
we prove that, under various assumptions on positive lower or upper densities
of $A$ and $B$, there is a high density set of witnesses to the piecewise
syndeticity of $A+B$. Most of the result are shown to hold more generally for
subsets of $\mathbb{Z}^{d}$. The key technical tool is a Lebesgue density
theorem for measure spaces induced by cuts in the nonstandard integers.

\end{abstract}
\maketitle

\section{Introduction and
preliminaries\label{Section: Introduction and preliminaries}}

\subsection{Sumsets and piecewise
syndeticity\label{Subsection: Sumsets and piecewise syndeticity}}

The earliest result on the relationships between density of sequences, sum or
difference sets, and syndeticity is probably Furstenberg's theorem mentioned
in \cite[Proposition 3.19]{furstenberg}: If $A$ has positive upper Banach
density, then $A-A$ is syndetic, i.e.\ has bounded gaps. \ The proof of the
theorem is essentially a pigeonhole argument.

In \cite{jin} Jin shows that if $A$ and $B$ are two subsets of $\mathbb{N}$
with positive upper Banach densities, then $A+B$ must be piecewise syndetic,
i.e.\ for some $m$, $A+B+[0,m]$ contains arbitrarily long intervals. \ Jin's
proof uses nonstandard analysis. \ In \cite{JK}, this result is extended to
abelian groups with tiling structures. In \cite{jin,JK} the question as to
whether this result can be extended to any countable amenable group is posed,
and in \cite{BBF} a positive answer to the above question is proven. It is
shown that if $A$ and $B$ are two subsets of a countable amenable group with
positive upper Banach densities, then $A\cdot B$ is piecewise Bohr, which
implies piecewise syndeticity. \ In fact, a stronger theorem is obtained in
the setting of countable abelian groups: A set $S$ is piecewise Bohr if and
only if $S$ contains the sum of two sets $A$ and $B$ with positive upper
Banach densities. \ Jin's theorem was generalized to arbitrary amenable groups
in \cite{dinasso-lupini}. \ At the same time, several new proofs of the
theorem in \cite{jin} have appeared. For example, an ultrafilter proof is
obtained in \cite{beiglock}. A more quantitative proof that includes a bound
based on the densities is obtained in \cite{dinasso2} by nonstandard methods,
and in \cite{dinasso} by elementary means.

However, there has not been any progress on extending the theorem in
\cite{jin} to lower asymptotic density or upper asymptotic density instead of
upper Banach density. Of course, if $A$ and $B$ have positive lower (upper)
asymptotic densities then they have positive Banach density, so $A+B$ must be
piecewise syndetic. \ In this paper we show that there is significant
uniformity to the piecewise syndetiticy in the sense that there are a large
density of points in the sumset with no gap longer than some fixed $m$.
\ Furthermore this result can be extended to all finite dimensions.
\ Specifically we show that if $A\subseteq\mathbb{Z}^{d}$ has positive
upper\ $d$-dimensional asymptotic density $\alpha$, and $B$ simply has
positive Banach density, then there exists a fixed $m$ such that for all $k$
the upper\ density of the set of elements $z$ in $\mathbb{Z}^{d}$ such that
$z+[-k,k]^{d}\subseteq A+B+[-m,m]^{d}$ is at least $\alpha$. \ For lower
density we show that the analogous conclusion must be slightly weakened as
follows: If $A\subseteq\mathbb{Z}^{d}$ has positive lower\ $d$-dimensional
asymptotic density $\alpha$, and $B$ has positive Banach density, then for any
$\epsilon>0$ there exists a fixed $m$ such that for all $k$ the lower\ density
of the set of elements $z$ in $\mathbb{Z}^{d}$ such that $z+[-k,k]^{d}%
\subseteq A+B+[-m,m]^{d}$ is at least $\alpha-\epsilon$. \ If both $A$ and $B$
have positive lower density in one dimension then, by using Mann's Theorem, we
show that the set of elements $z$ in $\mathbb{Z}$ such that $z+[-k,k]\subseteq
A+B+[-m,m]$ has a strong version of upper density of at least $\alpha+\beta$. \ 

The nonstandard methods used in this paper include a new Lebesgue Density
Theorem for \textquotedblleft cuts\textquotedblright\ in the nonstandard
integers. \ In \cite{KL} a quasi-order-topology, with respect to each additive
cut, was defined on a hyperfinite interval $[0,H]$ of integers. Motivated by
the duality\footnote{The ideal of null sets $\mathcal{N}$ is the collection of
all subsets of $\mathbb{R}$ with Lebesgue measure $0$ and the ideal of meager
sets $\mathcal{M}$ is the collection of all meager subsets of $\mathbb{R}$,
where a set is called meager if it is a countable union of nowhere dense sets.
$\mathcal{N}$ and $\mathcal{M}$ are dual ideals in the sense that $\mathbb{R}$
is the union of a meager set and a null set.} of the ideal of null sets and
the ideal of meager sets of real numbers, and the fact that the sum of two
sets with positive Lebesgue measure can never be meager (because it always
contains a non-empty open interval), a question was raised in \cite{KL}: Is
the sum of any two sets with positive Loeb measure in a hyperfinite interval
$[0,H]$ non-meager in the sense of the quasi-order-topology? \ A positive
answer to the question above led to Jin's result about piecewise syndeticity.
\ Here we study these cuts in $d$ dimensions and prove the following result:
\ If $H\in$ $^{\ast}\mathbb{N}\left\backslash \mathbb{N}\right.  $, $U$ is a
subset of $[1,H]$ that is closed under addition, $\mathcal{U}=(-U)\cup
\{0\}\cup(U),$ and $E$ is an internal subset of $[-H,H]^{d}$ then almost all
points $x$ in $E+\mathcal{U}^{d}$ are points of density in the sense that
\[
\liminf_{\nu>U}\mu_{x+[-\nu,\nu]^{d}}\left(  \left(  E+\mathcal{U}^{d}\right)
\cap(x+\left[  -\nu,\nu\right]  ^{d})\right)  =1,
\]
or, equivalently, to clarify the meaning of $\lim\inf$ in this setting:
\[
\sup_{\xi>U}\inf_{U<\nu<\xi}\mu_{x+[-\nu,\nu]^{d}}\left(  \left(
E+\mathcal{U}^{d}\right)  \cap(x+\left[  -\nu,\nu\right]  ^{d})\right)  =1,
\]
where $\mu_{x+[-\nu,\nu]^{d}}$ is the Loeb measure on $x+[-\nu,\nu]^{d}.$
\ Here, as in the rest of the paper, when we write that an element is greater
than an initial segment we mean that it is larger than every element in that
segment. \ For example $U<\nu$ means that for all $u\in U$, $u<\nu$. \ We use
the density theorem above in the case that $U=\mathbb{N}$ to obtain many of
the aforementioned standard results. \ 

A word about notation:\ \ In an effort to clarify standard vs.\ nonstandard
sets and elements, we will reserve $H,I,J,K,L,M,N$ for infinite hypernatural
numbers, while $\nu,\xi,\zeta$ will denote (possibly standard) hypernatural
numbers; Lower case letters denote elements of $\mathbb{Z}$ or $\mathbb{Z}%
^{d}$ and their nonstandard extensions; $A$ and $B$ will be reserved for
standard subsets of $\mathbb{N}$ (we do not include $0$ in $\mathbb{N}$),
$\mathbb{Z}$, or $\mathbb{Z}^{d}$; $E$,$R,S$, $T,X$ and $Y$ will be used for
subsets of $^{\ast}\mathbb{Z}^{d}$, with $E$ only used for internal sets. If
$\left(  a_{n}\right)  _{n\in\mathbb{N}}$ is a sequence, and $\nu$ is an
infinite hypernatural number, we denote by $a_{\nu}$ the value at $\nu$ of the
nonstandard extension of the sequence $\left(  a_{n}\right)  _{n\in\mathbb{N}%
}$. We use\ $\mu$ for measure, $\operatorname{d}$\ for density functions, and
$d$ for dimension. \ Here the values of $d$ are only natural numbers.
\ Despite these conventions the location of elements and sets is usually noted
at the time, at the risk of redundancy, but in the interest of clarity.

\subsection{Standard concepts of density and structure on
sequences\label{Subsection: Standard concepts of density and structure on sequences}%
}

In this paper we consider the following notions of density for a subset $A$ of
$\mathbb{Z}^{d}$:

\begin{itemize}
\item the \emph{lower (asymptotic) density}%
\[
\underline{\operatorname{d}}(A):=\liminf_{n\rightarrow\infty}\frac
{|A\cap\lbrack-n,n]^{d}|}{(2n+1)^{d}}\text{;}%
\]

\item the \emph{upper (asymptotic) density}%
\[
\overline{\operatorname{d}}(A):=\limsup_{n\rightarrow\infty}\frac
{|A\cap\lbrack-n,n]^{d}|}{(2n+1)^{d}}\text{;}%
\]

\item the \emph{Schnirelmann density}%
\[
\sigma(A):=\inf_{n}\frac{|A\cap\lbrack-n,n]^{d}|}{(2n+1)^{d}}\text{;}%
\]

\item the \emph{(upper) Banach density}%
\[
\operatorname{BD}(A):=\lim_{n\rightarrow\infty}\sup_{x\in\mathbb{Z}^{d}}%
\frac{|A\cap(x+[-n,n]^{d})|}{(2n+1)^{d}}\text{.}%
\]

\end{itemize}

In the particular case of $d=1$ these are the usual notions of density for
sequences of integers. It follows immediately from the definition that for any
$\epsilon>0$ there exists an $m\in\mathbb{N}$ such that
\[
\sigma(A\cup\lbrack-m,m]^{d})\geq\overline{\operatorname{d}}(A)-\epsilon
\text{.}%
\]
Moreover it is useful to note that if $\operatorname{BD}(A)>0$ then for any
$\epsilon>0$ there exists $m\in\mathbb{N}$ such that
\[
\operatorname{BD}(A+[-m,m]^{d})>1-\epsilon.
\]
When $d=1$ this is the content of Theorem 3.8 in \cite{hindman-density}.

We will refer to the following combinatorial notions of largeness for a subset
$A$ of $\mathbb{Z}^{d}$:

\begin{itemize}
\item $A$ is \emph{syndetic} iff there exists $m\in\mathbb{N}$ such that
$A+[-m,m]^{d}=\mathbb{Z}^{d}$;

\item $A$ is \emph{thick }iff there are arbitrarily large hypercubes
completely contained in $A$, i.e.\ for all $k\in\mathbb{N}$ there exists
$z\in\mathbb{Z}^{d}$ such that
\[
z+[-k,k]^{d}\subseteq A;
\]

\item $A$ is \emph{piecewise syndetic }iff there exists $m\in\mathbb{N}$ such
that $A+[-m,m]^{d}$ is thick, i.e. for all $k\in\mathbb{N}$ there exists
$z\in\mathbb{Z}^{d}$ such that
\[
z+[-k,k]^{d}\subseteq A+[-m,m]^{d}.
\]
Thus, $A$ is piecewise syndetic iff it is the intersection of a syndetic set
and a thick set.
\end{itemize}

While defining the densities on sets of the form $[-n,n]^{d}$ is natural, all
of our results involving the notion of upper or lower syndeticity can be
easily adapted to the setting where one considers arbitrary F\o lner
sequences. Of particular interest for all of our results is the case in which
$d=1$ where the interval $[-n,n]$ is replaced by $[1,n]$. \ This is the
classical setting for the study of densities of subsets of natural numbers.
\ To underscore the importance of that case and to improve clarity, almost all
of our examples are specific to this case, although all theorems and proofs
will be given in $d$ dimensions wherever possible.

It is not difficult to show that $\operatorname{BD}(A)=1$ iff $A$ is thick;
more precisely, if for some $k\in\mathbb{N}$ every cube $z+\left[
-k,k\right]  ^{d}$ is not contained in $A$, then $\operatorname{BD}%
(A)\leq\frac{(2k+1)^{d}-1}{(2k+1)^{d}}$. \ On the other hand, for every $r<1$
there exist sets of lower density at least $r$ that are not piecewise
syndetic. Indeed, if $n$ is sufficiently large and%
\[
B=\bigcup_{j=n}^{\infty}\bigcup_{x\in\mathbb{Z}^{d}\setminus\{0\}}\left(
(j!)x+\mathcal{[}1,(j-1)!]^{d}\right)  \text{,}%
\]
then $\mathbb{Z}^{d}\backslash B$ is an example of such a set.

\subsection{Nonstandard
preliminaries\label{Subsection: Nonstandard preliminaries}}

We use \emph{nonstandard analysis} to derive our results and we assume that
the reader is familiar with elementary nonstandard arguments. For an
introduction to nonstandard methods aimed specifically toward applications to
combinatorial number theory see \cite{jinintro}. \ Throughout this paper, we
always work in a countably saturated nonstandard universe.

We make extensive use of the concept of Loeb measure. \ Here we will always be
starting with the counting measure on some internal subset $E$ of $[-H,H]^{d}$
where $H$ is some element in $^{\ast}\mathbb{N}\left\backslash \mathbb{N}%
\right.  $. \ Often $E$ itself is $[-H,H]^{d}$, but it may also be a set of
the form $x+[-J,J]^{d}$ where $x\in\mathbb{Z}^{d}$ and $J\in$ $^{\ast
}\mathbb{N}\left\backslash \mathbb{N}\right.  $. \ For every internal $D$
contained in $E$, the measure of $D$ relative to $E$ is defined to be $\mu
_{E}(D):=\operatorname{st}(\frac{|D|}{|E|})$, where $\operatorname{st}$ is the
standard part mapping. \ This defines a finitely additive measure on the
algebra of internal subsets of $E$, which canonically extends to a countably
additive probability measure on the $\sigma$-algebra of \emph{Loeb measurable}
subsets of $E$, and we will also write $\mu_{E}$ for this extension. \ If $D$
is defined on a larger set than $E$ then we will write simply $\mu_{E}(D)$ for
$\mu_{E}(D\cap E)$.

We will make frequent use of the well-known proposition below, which gives
nonstandard equivalents for the standard density properties. \ Proofs are
included for convenience.

\begin{proposition}
\label{prop:nonstequivalents}If $A$ is a subset of $\mathbb{Z}^{d}$ then we
have the following nonstandard equivalents of the standard asymptotic densities:

\begin{enumerate}
\item if $\overline{\operatorname{d}}(A)\geq\alpha$ then for all $K\in$
$^{\ast}\mathbb{N}\left\backslash \mathbb{N}\right.  $ there exists an $H\in$
$^{\ast}\mathbb{N}\left\backslash \mathbb{N}\right.  $ such that $H<K$ and
$\mu_{[-H,H]^{d}}({}^{\ast}\!{A})\geq\alpha$. Conversely If there exists
$H\in$ $^{\ast}\mathbb{N}\left\backslash \mathbb{N}\right.  $ such that
$\mu_{[-H,H]^{d}}({}^{\ast}\!{A})\geq\alpha$ then $\overline{\operatorname{d}%
}(A)\geq\alpha$;

\item $\underline{\operatorname{d}}(A)\geq\alpha$ iff for all $H\in$ $^{\ast
}\mathbb{N}\left\backslash \mathbb{N}\right.  $ $\ \mu_{[-H,H]^{d}}({}^{\ast
}\!{A})\geq\alpha$;

\item If $\operatorname{BD}(A)\geq\alpha$ then for all $K\in$ $^{\ast
}\mathbb{N}\left\backslash \mathbb{N}\right.  $ there exists $J\in{}\left[
0,K\right]  \left\backslash \mathbb{N}\right.  $ and $x\in\lbrack-K,K]^{d}$
such that $\mu_{x+[-J,J]^{d}}({}^{\ast}\!{A})\geq\alpha$. Conversely if there
exists $J\in$ $^{\ast}\mathbb{N}\left\backslash \mathbb{N}\right.  $ and
$x\in$ $^{\ast}\mathbb{Z}^{d}$ such that $\mu_{x+[-J,J]^{d}}(^{\ast}%
\!A)\geq\alpha$, then $\operatorname{BD}(A)\geq\alpha$.
\end{enumerate}
\end{proposition}

\begin{proof}
$\,$

\begin{enumerate}
\item If $\overline{\operatorname{d}}(A)\geq\alpha$ then there exists a
sequence $n_{i}\rightarrow\infty$ such that for all $i\in\mathbb{N}$
\[
\frac{|A\cap\lbrack-n_{i},n_{i}]^{d}|}{(2n_{i}+1)^{d}}\geq\alpha-\frac{1}%
{i}\text{.}%
\]
Pick an infinite hypernatural number $J$ such that $n_{J}<K$, and observe that
$\mu_{\left[  -n_{J},n_{J}\right]  }({}^{\ast}\!{A})\geq\alpha$. Conversely
suppose that there exists an $H\in$ $^{\ast}\mathbb{N}\left\backslash
\mathbb{N}\right.  $ such that $\mu_{[-H,H]^{d}}({}^{\ast}\!{A})\geq\alpha$.
Given any $\epsilon>0$ and any $m\in\mathbb{N}$ one can deduce by transfer
that there is a natural number $n>m$ such that
\[
\frac{|A\cap\lbrack-n,n]^{d}|}{(2n+1)^{d}}\geq\alpha-\epsilon\text{.}%
\]
Therefore\ $\overline{\operatorname{d}}(A)\geq\alpha$.

\item $\underline{\operatorname{d}}(A)\geq\alpha$ iff for any $\epsilon>0$
there exists $n_{\epsilon}\in\mathbb{N}$ such that for all $n>n_{\epsilon}$
\[
\frac{|A\cap\lbrack-n,n]^{d}|}{(2n+1)^{d}}\geq\alpha-\epsilon.
\]
By transfer, this is true iff for all $H\in$ $^{\ast}\!\mathbb{N}%
\left\backslash \mathbb{N}\right.  $ and every standard $\epsilon>0$
\[
\frac{|{}^{\ast}\!A\cap[-H,H]^{d}|}{(2H+1)^{d}}\geq\alpha-\epsilon,
\]
\ which is equivalent to $\mu_{[-H,H]^{d}}({}^{\ast}\!{A})\geq\alpha$.

\item If $\operatorname{BD}(A)\geq\alpha$ then there exists a sequence\ $j_{i}%
$, with $j_{i}\rightarrow\infty$ and points $x_{i}\in\mathbb{Z}^{d}$ such that
for all $i\in\mathbb{N}$
\[
\frac{|A\cap(x_{i}+[-j_{i},j_{i}]^{d})|}{(2j_{i}+1)^{d}}\geq\alpha-1/i.
\]
We can pick an infinite hypernatural number $L$ such that $j_{L}<K$ and
$x_{L}<K$. \ We may now let $J=j_{L}$ and $x=x_{L}$. Conversely if there exist
$J\in$ $^{\ast}\mathbb{N}\left\backslash \mathbb{N}\right.  $ and $x\in$
$^{\ast}\mathbb{Z}^{d}$ such that $\mu_{x+[-J,J]^{d}}({}^{\ast}\!{A}%
)\geq\alpha$, then given any $\epsilon>0$ and any $m\in\mathbb{N}$ one can
deduce by transfer that there is are natural numbers $j$ and $x$ such that
$j>m$ and
\[
\frac{|A\cap(x+[-j,j]^{d})|}{(2j+1)^{d}}\geq\alpha-\epsilon\text{.}%
\]
This witnesses the fact that $\operatorname{BD}(A)\geq\alpha$.
\end{enumerate}
\end{proof}

\subsection{Acknowledgements}

This work was partly completed during a week-long meeting at the American
Institute for Mathematics on June 3-7, 2013 as part of the SQuaRE (Structured
Quartet Research Ensemble) project \textquotedblleft Nonstandard Methods in
Number Theory.\textquotedblright\ The authors would like to thank the
Institute for this opportunity and for the warm hospitality they enjoyed
during their stay.

\section{Points of Density\label{Section: Points of Density}}

If $E$ is an internal subset of $^{\ast}\mathbb{Z}^{d}$ and $x\in{}^{\ast
}\mathbb{Z}^{d}$ define%
\begin{align*}
d_{E}(x)  &  =\liminf_{\nu>\mathbb{N}}\mu_{x+[-\nu,\nu]^{d}}\left(  \left(
E+\mathbb{Z}^{d}\right)  \cap(x+\left[  -\nu,\nu\right]  ^{d})\right) \\
&  =\sup_{H>\mathbb{N}}\inf_{\mathbb{N}<\nu<H}\mu_{\lbrack-\nu,\nu]^{d}%
}\left(  \left(  \left(  E-x\right)  +\mathbb{Z}^{d}\right)  \cap\left[
-\nu,\nu\right]  ^{d}\right)  \text{.}%
\end{align*}

\begin{definition}
If $d_{E}(x)=1$ we say that $x$ is a \emph{point of density }of $E$, and we
write $\mathcal{D}_{E}$ for the set of points of density of $E$. \ 
\end{definition}

We note that $\mathcal{D}_{E}$ is not, in general, internal and that
$\mathcal{D}_{E}+\mathbb{Z}^{d}\subseteq\mathcal{D}_{E}$. It is easy to see by
countable saturation that for $r\in\left[  0,1\right]  $ and $x\in{}^{\ast
}\mathbb{Z}^{d}$, $d_{E}\left(  x\right)  \geq r$ if and only if there is
$H>\mathbb{N}$ such that for every $\mathbb{N}<\nu<H$%
\[
\mu_{x+[-\nu,\nu]^{d}}\left(  \left(  E+\mathbb{Z}^{d}\right)  \cap(x+\left[
-\nu,\nu\right]  ^{d})\right)  \geq r\text{.}%
\]

Theorem \ref{thm:LDTtheorem1} is our main result concerning points of density,
and can be regarded as an analogue of the Lebesgue density theorem. \ It
implies, in particular, that an internal set of positive Loeb measure relative
to some interval always has points of density.

\begin{theorem}
\label{thm:LDTtheorem1}If $E\subseteq[-H,H]^{d}$ is internal then
$\mathcal{D}_{E}$ is Loeb measurable, and $\mu_{[-H,H]^{d}}(\mathcal{D}%
_{E})=\mu_{[-H,H]^{d}}(E+\mathbb{Z}^{d})$.
\end{theorem}

In Section \ref{Section: A Lebesgue density theorem for nonstandard cuts} we
will consider a similar notion\ of density point for arbitrary cuts and prove,
in Corollary \ref{Corollary: LDT}, a more general version of Theorem
\ref{thm:LDTtheorem1}, which can be regarded as a Lebesgue density theorem for
measure spaces induced by cuts in the nonstandard integers.

It is worth noting that the Loeb measure in the usual sense does not satisfy a
similar analogue of the Lebesgue density theorem. \ For example the set of
even numbers smaller than $H$ has relative Loeb measure $1\left/  2\right.  $
on every infinite interval. Theorem \ref{thm:LDTtheorem1} says that if we
identify points that are a finite distance apart, then the Loeb measure on
that quotient space does have a density theorem very similar to that of the
Lebesgue measure. \ Proposition
\ref{prop:points of density achieve full measure} highlights a way in which
the theorem is even stronger than it is for Lebesgue measure, where sets might
have no interval about a point of density that actually achieves relative
measure $1$.

\begin{proposition}
\label{prop:points of density achieve full measure}If $E$ is an internal
subset of $[-H,H]^{d}$ then $d_{E}(x)=1$ if and only if there exists
$\nu>\mathbb{N}$ such that for all $\mathbb{N}<K<\nu$,
\[
\mu_{x+\left[  -K,K\right]  ^{d}}\left(  \left(  E+\mathbb{Z}^{d}\right)
\cap(x+\left[  -K,K\right]  ^{d})\right)  =1.
\]

\end{proposition}

\begin{proof}
Suppose that $d_{E}(x)=1$. \ Then from the definition there exist $\nu
_{j}>\mathbb{N}$ such that for all $\mathbb{N}<K<\nu_{j}$,%
\[
\mu_{x+\left[  -K,K\right]  ^{d}}\left(  \left(  E+\mathbb{Z}^{d}\right)
\cap(x+\left[  -K,K\right]  ^{d})\right)  \geq1-1/j.
\]
\ By countable saturation we can find $\nu>\mathbb{N}$ less than all the
$\nu_{j}$. The converse is immediate.
\end{proof}

We can characterize points of density in terms of standard density functions
on subsets of $\mathbb{Z}^{d}$ that are centered around nonstandard points.
\ In order to do so we need to approximate the \textquotedblleft%
$+\mathbb{Z}^{d}$\textquotedblright\ part of the statement by considering how
a set intersects with larger and larger \textquotedblleft
blocks.\textquotedblright

Given a subset $A$ of $\mathbb{Z}^{d}$ (or an internal subset of $^{\ast
}\mathbb{Z}^{d})$ and $n\in\mathbb{N}$ (or $^{\ast}\mathbb{N}$) we define the
\emph{n-block sets }$A_{[n]}$ and $A^{[n]}$ \emph{of }$A$ by
\[
x\in A_{[n]}\text{ iff }(nx+[0,n-1]^{^{d}})\cap A\neq\varnothing.
\]
and
\[
A^{[n]}=nA_{[n]}+[0,n-1]^{d}.
\]

We note that $A_{[n]}$ and $A^{[n]}$ have the same asymptotic densities, but
are \textquotedblleft scaled\textquotedblright\ differently, with $A^{[n]}$ on
the same scale as $A$. \ In fact, $A\subseteq A^{[n]}$, which consists of a
union of $[0,n-1]^{d}$ blocks whose position is determined by the elements of
$A_{[n]}$.\ \ More specifically
\[
x\in A_{[n]}\text{ iff }(nx+[0,n-1]^{^{d}})\cap A\neq\varnothing\text{ iff
}nx+[0,n-1]^{^{d}}\subseteq A^{[n]}.
\]
Thus, blocks of the form $nx+[0,n-1]^{d}$ containing any element of $A$ are
\textquotedblleft completely filled in\textquotedblright\ to form $A^{[n]}$.

If $E$ is internal the set $E+\mathbb{Z}^{d}$ is, in general, external, but
its properties can often be approximated by the internal sets $E^{[n]}$ or
$E+[-n,n]^{^{d}}$ for large finite $n$ or \textquotedblleft
small\textquotedblright\ elements of $^{\ast}\mathbb{N}\left\backslash
\mathbb{N}\right.  $. \ The following observations are all stratightforward
and will be useful:

\begin{itemize}
\item If $j\in\mathbb{N}$ and $J\in$ $^{\ast}\mathbb{N}\left\backslash
\mathbb{N}\right.  $ and $E\subseteq[-H,H]^{d}$ is internal, then
\[
\mu_{[-H,H]^{d}}(E^{[j]})\leq\mu_{[-H,H]^{d}}(E+\mathbb{Z}^{d})\leq
\mu_{[-H,H]^{d}}(E^{[J]}).
\]

\item For any internal $E\subseteq[-H,H]^{d}$
\[
\lim_{i\rightarrow\infty}(\mu_{[-H,H]^{d}}(E^{[i]}))=\lim_{i\rightarrow\infty
}(\mu_{[-H,H]^{d}}(E+[-i,i]^{d}))=\mu_{[-H,H]^{d}}(E+\mathbb{Z}^{d}).
\]

\item For $A\subseteq\mathbb{Z}^{d}$%
\[
\lim_{i\rightarrow\infty}(\overline{\operatorname{d}}(A_{[i]}))=\lim
_{i\rightarrow\infty}(\overline{\operatorname{d}}(A^{[i]}))=\lim
_{i\rightarrow\infty}(\overline{\operatorname{d}}(A+[-i,i]^{d})),
\]
and%
\[
\lim_{i\rightarrow\infty}(\underline{\operatorname{d}}(A_{[i]}))=\lim
_{i\rightarrow\infty}(\underline{\operatorname{d}}(A^{[i]}))=\lim
_{i\rightarrow\infty}(\underline{\operatorname{d}}(A+[-i,i]^{d})).
\]

\end{itemize}

The limits in the statement above always exist, even though it is not true
that $i<j$ implies that the upper and lower densities of $A^{[j]}$ are at
least those of $A^{[i]}$. \ For example in one dimension, if $A$ consists of
all numbers that are 0,1, or 2 mod 6 then $\underline{\operatorname{d}%
}(A_{[2]})=\overline{\operatorname{d}}(A_{[2]})=2/3$ while
$\underline{\operatorname{d}}(A_{[3]})=\overline{\operatorname{d}}%
(A_{[3]})=1/2.$ \ However, it is easy to see that for all $i$ \ and for all
$\epsilon>0$ there exists $l$ such that for all $j>l$,
$\underline{\operatorname{d}}(A_{[j]})\geq\underline{\operatorname{d}}%
(A_{[i]})-\epsilon$ and $\overline{\operatorname{d}}(A_{[j]})\geq
\overline{\operatorname{d}}(A_{[i]})-\epsilon$. \ Thus $\lim_{i\rightarrow
\infty}(\underline{\operatorname{d}}(A_{[i]}))$ and $\lim_{i\rightarrow\infty
}(\overline{\operatorname{d}}(A_{[i]}))$ always exist. \ Of course,
$\lim_{i\rightarrow\infty}(\overline{\operatorname{d}}(A+[-i,i]^{d}))$ clearly
always exists.

\begin{proposition}
Let $r$ be a standard real number between $0$ and $1$, and $E$ be an internal
subset of $^{\ast}\mathbb{Z}^{d}$. If $x\in{}^{\ast}\mathbb{Z}^{d}$, then
$d_{E}\left(  x\right)  >r$ if and only if $\sigma((E-x)^{[n]}\cap
\mathbb{Z}^{d})>r$ for some $n\in\mathbb{N}$. In particular $x$ is a point of
density if and only if for every $r\in\left(  0,1\right)  $ there is
$n\in\mathbb{N}$ such that $\sigma\left(  \left(  E-x\right)  ^{\left[
n\right]  }\cap\mathbb{Z}^{d}\right)  >r$.
\end{proposition}

\begin{proof}
Suppose that $\sigma((E-x)^{[n]}\cap\mathbb{Z}^{d})\leq r$ for every
$n\in\mathbb{N}$. Fix an arbitrary strictly positive standard real number
$\epsilon$, and pick a sequence $\left(  l_{n}\right)  _{n\in\mathbb{N}}$ in
$\mathbb{N}$ such that
\[
\frac{\left\vert (E-x)^{[n]}\cap\lbrack-l_{n},l_{n}]^{d}\right\vert }%
{(2l_{n}+1)^{d}}<r+\epsilon\text{\label{Eq: density estimate}}%
\]
for every $n\in\mathbb{N}$.\ Observe that $\left(  l_{n}\right)
_{n\in\mathbb{N}}$ is a divergent sequence of natural numbers. Let $H$ be an
arbitrary infinite hypernatural number, and pick an infinite hypernatural
number $\nu$ such that $l_{\nu}<H$. We have that%
\begin{align*}
\mu_{\lbrack-l_{\nu},l_{\nu}]^{d}}\left(  E-x+\mathbb{Z}^{d}\right)   &
\leq\mu_{\lbrack-l_{\nu},l_{\nu}]^{d}}\left(  (E-x)^{[\nu]}\right) \\
&  \approx\frac{\left\vert (E-x)^{[\nu]}\cap\left[  -l_{\nu},l_{\nu}\right]
^{d}\right\vert }{(2l_{\nu}+1)^{d}}<r+\epsilon.
\end{align*}
Since $\nu$ could be an arbitrarily small element in $^{\ast}\mathbb{N}%
\left\backslash \mathbb{N}\right.  $, this shows that $d_{E}(x)\leq
r+\epsilon$. Being this true for every standard real number $\epsilon$,
$d_{E}\left(  x\right)  \leq r$. Conversely suppose that $\sigma
((E-x)^{\left[  n\right]  })\geq r+\epsilon$ for some strictly positive
standard real number $\epsilon$. Thus for every $k\in\mathbb{N}$%
\[
\frac{\left\vert \left(  E-x\right)  ^{\left[  n\right]  }\cap\left[
-k,k\right]  ^{d}\right\vert }{\left(  2k+1\right)  ^{d}}\geq r+\epsilon
\text{.}%
\]
If $H$ is an infinite hypernatural number, then by overspill there is an
infinite $\nu<H$ such that%
\[
\mu_{\left[  -\nu,\nu\right]  ^{d}}\left(  \left(  E-x\right)  +\mathbb{Z}%
^{d}\right)  \geq\frac{\left\vert \left(  E-x\right)  ^{\left[  n\right]
}\cap\left[  -\nu,\nu\right]  ^{d}\right\vert }{\left(  2\nu+1\right)  ^{d}%
}\geq r+\epsilon\text{,}%
\]
witnessing the fact that $d_{E}\left(  x\right)  \geq r+\epsilon$.
\end{proof}

\begin{definition}
If $E$ is an internal subset of $^{\ast}\mathbb{Z}^{d}$ and $x\in{}^{\ast
}\mathbb{Z}^{d}$ we say that $x$ is a \emph{point of syndeticity }of $E$ iff
there exists a finite $m$ such that $x+\mathbb{Z}^{d}\subseteq E+[-m,m]^{d}$. \ 
\end{definition}

Equivalently, since $E+[-m,m]^{d}$ is internal, $x$ is a point of syndeticity
of $E$ iff there exists $m\in\mathbb{N}$ and $K\in$ $^{\ast}\mathbb{N}%
\left\backslash \mathbb{N}\right.  $ such that $x+[-K,K]^{d}\subseteq
E+[-m,m]^{d}$. \ We will write $\mathcal{S}_{E}$ for the set of all points of
syndeticity of $E$. \ Like $\mathcal{D}_{E}$, $\mathcal{S}_{E}$ is, in
general, not internal, and $\mathcal{S}_{E}+\mathbb{Z}^{d}\subseteq
\mathcal{S}_{E}$.

\begin{theorem}
\label{thm:density plus density is syndetic}Suppose that $X,Y$ are internal
subsets of ${}^{\ast}\mathbb{Z}^{d}$ and $a,b\in{}^{\ast}\mathbb{Z}^{d}$. If
$d_{X}\left(  a\right)  =1$ and $d_{Y}\left(  b\right)  =1$ then $a+b$ is a
point of syndeticity of $X+Y.$
\end{theorem}

\begin{proof}
By Proposition \ref{prop:points of density achieve full measure} there exists
$\nu>\mathbb{N}$ such that%
\[
\mu_{\left[  -\nu,\nu\right]  ^{d}}\left(  \left(  X-a+\mathbb{Z}^{d}\right)
\cap\left[  -\nu,\nu\right]  ^{d}\right)  =1
\]
and%
\[
\mu_{\left[  -\nu,\nu\right]  ^{d}}\left(  \left(  -Y+b+\mathbb{Z}^{d}\right)
\cap\left[  -\nu,\nu\right]  ^{d}\right)  =1\text{.}%
\]
If $\xi\in\left[  -\frac{\nu}{2},\frac{\nu}{2}\right]  ^{d}$ then%
\[
\mu_{\left[  -\nu,\nu\right]  ^{d}}\left(  \left(  -Y+b+\xi+\mathbb{Z}%
^{d}\right)  \cap\left[  -\nu,\nu\right]  ^{d}\right)  \geq\frac{1}{2^{d}}%
\]
and hence%
\[
\left(  -Y+b+\xi+\mathbb{Z}^{d}\right)  \cap\left(  X-a+\mathbb{Z}^{d}\right)
\neq\varnothing\text{.}%
\]
This means that there are $x\in X$, $y\in Y$ and $u,v\in\mathbb{Z}^{d}$ such
that%
\[
a+b+\xi+u=x+y+v.
\]
This shows that%
\[
a+b+\left[  -\frac{\nu}{2},\frac{\nu}{2}\right]  ^{d}\subset X+Y+\mathbb{Z}%
^{d}\text{.}%
\]

Thus
\[
a+b+\left[  -\frac{\nu}{2},\frac{\nu}{2}\right]  ^{d}%
\]
is contained in the increasing union of internal sets
\[
\bigcup_{m\in\mathbb{N}}(X+Y+[-m,m]^{d})\text{.}%
\]

By saturation, it follows that
\[
a+b+\left[  -\frac{\nu}{2},\frac{\nu}{2}\right]  ^{d}\subseteq X+Y+[-m,m]^{d}%
\]
for some $m\in\mathbb{N}$. \ Now%
\[
a+b+\mathbb{Z}^{d}\subseteq a+b+\left[  -\frac{\nu}{2},\frac{\nu}{2}\right]
^{d}\subseteq X+Y+[-m,m]^{d},
\]
and hence $a+b$ is a point of syndeticity of $X+Y$.
\end{proof}

\begin{proposition}
\label{prop:syndeticimpliesresult}Let $E$ be an internal subset of
$[-H,H]^{d}$. \ Then $\mathcal{S}_{E}$ is $\mu_{[-H,H]^{d}}$-measurable.
\ Moreover, if $\mu(\mathcal{S}_{E})=\alpha>0$ then for all (standard)
$\epsilon>0$ there exists a (standard) $m\in\mathbb{N}$ such that for all
(standard) $k\in\mathbb{N}$,
\[
\mu_{[-H,H]^{d}}(\{z\in[-H,H]^{d}:z+[-k,k]^{d}\subseteq E+[-m,m]^{d}%
\})\geq\alpha-\epsilon.
\]

\end{proposition}

\begin{proof}
For each $i\in\mathbb{N}$ let
\[
\mathcal{S}_{E}^{i}=\left\{  x\in\text{ }[-H,H]^{d}:\text{ }x+\mathbb{Z}%
^{d}\subseteq E+[-i,i]^{d}\right\}  .
\]

Then
\[
\mathcal{S}_{E}=\bigcup_{i=1}^{\infty}\mathcal{S}_{E}^{i}.
\]

Each $\mathcal{S}_{E}^{i}$ is measurable since
\[
\mathcal{S}_{E}^{i}=\bigcap_{z\in\mathbb{Z}^{d}}(E+[-i,i]^{d}+z),
\]
and so is a countable intersection of internal sets. \ This shows that
$\mathcal{S}_{E}$ is measurable, as it is a countable union of measurable
sets. \ 

By countable additivity of the Loeb measure, and the fact that the
$\mathcal{S}_{E}^{i}$ form a nested sequence of sets, there must exist an $m$
such that
\[
\mu_{[-H,H]^{d}}\left(  \mathcal{S}_{E}^{m}\right)  \geq\alpha-\epsilon.
\]
This $m$ must now satisfy the statement, since for any $k$,
\[
\{z\in[-H,H]^{d}:z+[-k,k]^{d}\subseteq E+[-m,m]^{d}\}
\]
is an internal set that contains $\mathcal{S}_{E}^{m}$.
\end{proof}

The proposition above is false if we replace the conclusion
\[
\mu_{[-H,H]^{d}}(\{z\in[-H,H]^{d}:z+[-k,k]^{d}\subseteq E+[-m,m]^{d}%
\})\geq\alpha-\epsilon
\]
with
\[
\mu_{[-H,H]^{d}}(\{z\in[-H,H]^{d}:z+[-k,k]^{d}\subseteq E+[-m,m]^{d}%
\})\geq\alpha,
\]
as the example below shows (in one dimension). In this example the gaps are
arbitrarily large, but only on relatively small intervals.

\begin{example}
\label{ex:we need the epsilon}We define a standard sequence $A$ by:
\[
A=\bigcup_{j=1}^{\infty}\left(  \bigcup_{i=1}^{j}(i\cdot\mathbb{N)}\cap\left[
2^{j}-2^{j-i},2^{j}-2^{j-i-1}\right]  \right)  .
\]
Then it is easy to see that for any infinite $H$, $\mu_{[1,H]}(\mathcal{S}%
_{{}^{\ast}\!{A}})=1$, but that for any given natural number $m$,
\[
\mu_{[1,H]}(\{a\in[1,H]:[a,a+k]\subset{}^{\ast}\!A+[0,m]\})\leq1-\frac
{1}{2^{m}}.
\]

\end{example}

\section{Upper syndeticity and
sumsets\label{Section: Upper syndeticity and sumsets}}

Supose that $\alpha$ is a positive real number less than or equal to $1$. We
say that a subset $A$ of $\mathbb{Z}^{d}$ is:

\begin{itemize}
\item \emph{lower syndetic of level }$\alpha$\emph{ }iff there exists a
natural number $m\in\mathbb{N}$ such that for all $k\in\mathbb{N}$,%
\[
\underline{\operatorname{d}}(\{z\in\mathbb{Z}^{d}:z+[-k,k]^{d}\subseteq
A+[-m,m]^{d}\})\geq\alpha\text{;}%
\]

\item \emph{upper syndetic of level }$\alpha$\emph{ }iff there exists a
natural number $m\in\mathbb{N}$ such that for all $k\in\mathbb{N}$,
\[
\overline{\operatorname{d}}(\{z\in\mathbb{Z}^{d}:z+[-k,k]^{d}\subseteq
A+[-m,m]^{d}\})\geq\alpha\text{;}%
\]

\item \emph{strongly upper syndetic of level }$\alpha$ iff for any infinite
sequence $S\subseteq\mathbb{N}$, there exists $m\in\mathbb{N}$ such that for
any $k\in\mathbb{N}$
\[
\limsup_{n\in S}\frac{1}{(2n+1)^{d}}\left\vert \left\{  z\in\lbrack
-n,n]^{d}:z+[-k,k]^{d}\subseteq A+[-m,m]^{d}\right\}  \right\vert \geq
\alpha\text{.}%
\]

\end{itemize}

In accordance with previous definitions, if any of the above holds with $m=0$
we may replace the word \textquotedblleft syndetic\textquotedblright\ with the
word \textquotedblleft thick.\textquotedblright\ \ Thus a subset $A$ of
$\mathbb{Z}^{d}$ is:

\begin{itemize}
\item \emph{lower thick of level }$\alpha$\emph{ }iff for all $k\in\mathbb{N}%
$,
\[
\underline{\operatorname{d}}(\{z\in\mathbb{Z}^{d}:z+[-k,k]^{d}\subseteq
A\})\geq\alpha\text{;}%
\]

\item \emph{upper thick of level }$\alpha$\emph{ }iff for all $k\in\mathbb{N}%
$,
\[
\overline{\operatorname{d}}(\{z\in\mathbb{Z}^{d}:z+[-k,k]^{d}\subseteq
A\})\geq\alpha.
\]

\end{itemize}

We note that there is no need for the notion of \textquotedblleft strongly
upper thick,\textquotedblright\ since it would be equivalent to that of
\textquotedblleft lower thick.\textquotedblright\ \ We also note that lower
syndeticity of level $\alpha$ implies strong upper syndeticity of level
$\alpha$, which in turn is stronger than upper syndeticity of level $\alpha$.
It is also trivial that a lower thick set of level $\alpha$ is, in particular,
lower syndetic of level $\alpha$; the same fact holds for upper syndeticity of
level $\alpha$.

The set $C$ given in Example \ref{ex:optimal} in the next section shows that
\textquotedblleft strong upper syndeticity of level $\alpha$" is a notion that
lies strictly between lower and upper syndeticity of level $\alpha.$ \ In that
example the set $C$ is upper syndetic of level $1$, strongly upper syndetic of
level $\frac{1}{2}$, but not of level $\alpha$ if $\alpha>\frac{1}{2}$, and is
not lower syndetic of level $\alpha$ for any $\alpha>0$.

Observe that replacing the upper density with the Banach density in the
definition of upper syndetic of level $\alpha$ would make the notion
trivialize, since every piecewise syndetic set would satisfy that condition
with $\alpha=1$.

\begin{proposition}
\label{prop:unchanged measure proposition}Let $H\in$ $^{\ast}\mathbb{N}%
\left\backslash \mathbb{N}\right.  $, and $S$ be a measurable subset of
$[-H,H]^{d}$ with the property that $\mu_{[-H,H]^{d}}(S+\mathbb{Z}^{d}%
)=\mu_{[-H,H]^{d}}(S)=\alpha>0.$ \ Then for each standard $k\in\mathbb{N}$
\[
\mu_{[-H,H]^{d}}(\{x\in\text{ }S:x+[-k,k]^{d}\subseteq S\})=\alpha.
\]

\end{proposition}

\begin{proof}
Define%
\[
\gamma=\mu_{\left[  -H,H\right]  }\left(  \left\{  x\in S:x+\left[
-k,k\right]  ^{d}\nsubseteq S\right\}  \right)  \text{.}%
\]
Observe that%
\[
\mu_{\left[  -H,H\right]  ^{d}}\left(  S+\mathbb{Z}^{d}\right)  \geq
\mu_{\left[  -H,H\right]  }\left(  S+\left[  -k,k\right]  ^{d}\right)  \geq
\mu_{\left[  -H,H\right]  }\left(  S\right)  +\frac{\gamma}{\left(
2k+1\right)  ^{d}}\text{.}%
\]
Here the $(2k+1)^{d}$ term comes from the fact that a single point in
\[
\left(  S+[-k,k]^{d}\right)  \left\backslash S\right.
\]
witnesses that $x+\left[  -k,k\right]  ^{d}\nsubseteq S$ for at most $\left(
2k+1\right)  ^{d}$ elements $x$ of $S$. Since
\[
\mu_{\left[  -H,H\right]  ^{d}}\left(  S+\mathbb{Z}^{d}\right)  =\mu_{\left[
-H,H\right]  ^{d}}\left(  S\right)
\]
by assumption, this implies that $\gamma=0$. The conclusion follows.
\end{proof}

Corollary \ref{measure 1 syndeticity} is straightforward but will be useful,
and follows immediately from the previous result.

\begin{corollary}
\label{measure 1 syndeticity}Let $H\in$ $^{\ast}\mathbb{N}\left\backslash
\mathbb{N}\right.  $, and $S$ be a measurable subset of $[-H,H]^{d}$ with the
property that $\mu_{[-H,H]^{d}}(S)=1$. \ Then for each standard $k\in
\mathbb{N}$
\[
\mu_{[-H,H]^{d}}(\{x\in\text{ }S:x+[-k,k]^{d}\subseteq S\})=1.
\]

\end{corollary}

\begin{proposition}
\label{implies thick}Let $A$ be a subset of $\mathbb{Z}^{d}$. \ If
$\lim_{i\rightarrow\infty}(\overline{\operatorname{d}}(A^{[i]}))=\overline
{\operatorname{d}}(A)=\alpha>0$ then $A$ is upper thick of level $\alpha$.
\end{proposition}

\begin{proof}
By Proposition \ref{prop:nonstequivalents} there exists an $H\in$ $^{\ast
}\mathbb{N}\left\backslash \mathbb{N}\right.  $ such that
\[
\mu_{\lbrack-H,H]^{d}}({}^{\ast}\!{A})=\alpha,
\]
and since $\lim_{i\rightarrow\infty}(\overline{\operatorname{d}}%
(A^{[i]}))=\alpha$ it must be that for all $i\in\mathbb{N}$
\[
\mu_{\lbrack-H,H]^{d}}({}^{\ast}\!{A}^{[i]})=\alpha,
\]
for if there were any finite $i$ and any\ $H\in$ $^{\ast}\mathbb{N}%
\left\backslash \mathbb{N}\right.  $ for which $\mu_{\lbrack-H,H]^{d}}%
({}^{\ast}\!A^{[i]})>\alpha,$ that would imply that $\lim_{i\rightarrow\infty
}(\overline{\operatorname{d}}(A^{[i]}))>\alpha$. \ We now have that for all
$j\in\mathbb{N}$ there exists $i_{j}>j$ such that
\[
\alpha-1/j\leq\frac{\left\vert {}^{\ast}\!{A}^{[i_{j}]}\cap\lbrack
-H,H]^{d}\right\vert }{(2H+1)^{d}}\leq\alpha+1/j.
\]
By overspill there exists a $J$ in $[1,H]$ such that $i_{J}/H$ is
infinitesimal and
\[
\alpha-1/J\leq\frac{\left\vert {}^{\ast}\!{A}^{[i_{J}]}\cap\lbrack
-H,H]^{d}\right\vert }{(2H+1)^{d}}\leq\alpha+1/J,
\]
so that
\[
\mu_{\lbrack-H,H]^{d}}({}^{\ast}\!{A}^{[i_{J}]})=\alpha.
\]
The set ${}^{\ast}\!{A}^{[i_{J}]}$ consists of a union of hypercubes of the
form $(i_{J})x+[0,i_{J}-1]^{^{d}}$. \ Let $N=\left\lfloor H/i_{J}\right\rfloor
$, and
\[
K=\left\vert \left\{  x\in\lbrack-N,N]^{d}:(i_{J})x+[0,i_{J}-1]^{^{d}%
}\subseteq\text{ }{}^{\ast}\!{A}^{[i_{J}]}\right\}  \right\vert .
\]
Then, since $i_{J}$ is infinitesimal compared to $H$, and every hypercube of
the form $(i_{J})x+[0,i_{J}-1]^{^{d}}$ is either completely contained in
${}^{\ast}\!{A}^{[i_{J}]}$ or is disjoint from it we may conclude that
\[
\operatorname{st}\left(  \frac{K}{\left(  2N+1\right)  ^{d}}\right)
=\mu_{\lbrack-H,H]^{d}}({}^{\ast}\!{A}^{[i_{J}]})=\alpha.
\]
But $K$ is also equal to
\[
\left\vert \left\{  x\in\lbrack-N,N]^{d}:(i_{J})x+[0,i_{J}-1]^{^{d}}\cap
{}^{\ast}\!{A}\right\}  \neq\varnothing\right\vert .
\]
Since $\mu_{\lbrack-H,H]^{d}}({}^{\ast}\!{A})=\alpha$, this implies that on
almost every such cube
\[
\mu_{(i_{J})x+[0,i_{J}-1]^{^{d}}}({}^{\ast}\!{A})=1.
\]
By Corollary \ref{measure 1 syndeticity}, for each standard $k\in\mathbb{N}$
\[
\mu_{(i_{J})x+[0,i_{J}-1]^{^{d}}}(\{x\in\text{ }{}^{\ast}\!{A}:x+[-k,k]^{d}%
\subseteq{}^{\ast}\!{A}\})=1.
\]
Summing over all the $N$ blocks that intersect ${}^{\ast}\!{A}$ now yields the
desired result.
\end{proof}

The analogous result does not hold for lower density. \ There exist sets
$A\subseteq\mathbb{Z}^{d}$ such that $\lim_{i\rightarrow\infty}%
(\underline{\operatorname{d}}(A^{[i]}))=\underline{\operatorname{d}}%
(A)=\alpha>0$ but $A$ is not lower thick of level $\beta$ for any $\beta>0$.
\ See the remarks after Example \ref{ex:upper density theorem best possible}
below. \ However the proof above can easily be adapted to show that if
$\underline{\operatorname{d}}(A)=\alpha>0$ and $\lim_{i\rightarrow\infty
}(\overline{\operatorname{d}}(A^{[i]}))=\alpha$, then $A$ is lower thick of
level $\alpha$.

\begin{theorem}
\label{upper density sum theorem}Let $A$ and $B$ be subsets of $\mathbb{Z}%
^{d}$ with the property that $\overline{\operatorname{d}}(A)=\alpha>0$ and
$\operatorname{BD}(B)>0$. \ Then $A+B$ is upper syndetic of level $\alpha$.
\end{theorem}

\begin{proof}
If $\lim_{i\rightarrow\infty}(\overline{\operatorname{d}}(A^{[i]}))=\alpha$
then by Proposition \ref{implies thick} $A$ itself is upper thick of level
$\alpha$, so $A+B$ is certainly upper thick of level $\alpha$.

So, it suffices to assume that there exists $i\in\mathbb{N}$ such that
$\overline{\operatorname{d}}(A^{[i]})>\alpha$. \ This implies that there
exists $H\in$ $^{\ast}\mathbb{N}\left\backslash \mathbb{N}\right.  $ such
that
\[
\mu_{[-H,H]^{d}}({}^{\ast}\!{A}+\mathbb{Z}^{d})\geq\mu_{\lbrack-H,H]^{d}}%
({}^{\ast}\!{A}^{[i]})>\alpha.
\]
Let $\epsilon>0$ be less than $\mu_{[-H,H]^{d}}({}^{\ast}\!{A}+\mathbb{Z}%
^{d})-\alpha$. \ Since $\operatorname{BD}(B)>0$ there exist arbitrarily large
standard $j$ such that for some $k\in\mathbb{Z}^{d}$
\[
\frac{\left\vert B\cap(k+[-j,j]^{d})\right\vert }{(2j+1)^{d}}%
>\operatorname{BD}(B)/2.
\]
Let $\nu\in$ $[1,H]\left\backslash \mathbb{N}\right.  $ be such that $\nu/H$
is infinitesimal. \ By Proposition \ref{prop:nonstequivalents} there exists
$J\in$ $^{\ast}\mathbb{N}\left\backslash \mathbb{N}\right.  $ and\ $k\in
[-H,H]^{d}$ such that
\[
\frac{\left\vert {}^{\ast}B\cap(k+[-J,J]^{d})\right\vert }{(2J+1)^{d}%
}>\operatorname{BD}(B)/2,
\]
and $\left\vert k\right\vert $ and $J$ are less than $\nu$. \ Along with
Theorem \ref{thm:LDTtheorem1}, this shows that there is a point of density $b$
of $^{\ast}B$ such that $\frac{\left\vert b\right\vert }{H}$ is infinitesimal.
\ Then by Proposition \ref{thm:density plus density is syndetic} every point
in $b+\mathcal{D}_{{}^{\ast}\!{A}}+\mathbb{Z}^{d}$ is a syndetic point of
${}^{\ast}\!(A+B)$, and since $\left\vert b\right\vert $ is infinitesimal to
$H$, almost all elements of $b+(\mathcal{D}_{{}^{\ast}\!{A}}{}\cap$
$[-H,H]^{d})$ are in $[-H,H]^{d}$. \ 

By Theorem \ref{thm:LDTtheorem1} we know that
\[
\mu_{[-H,H]^{d}}(\mathcal{D}_{{}^{\ast}\!{A}})=\mu_{[-H,H]^{d}}({}^{\ast}%
\!{A}+\mathbb{Z}^{d})>\alpha+\epsilon\text{,}\
\]
so that
\[
\mu_{[-H,H]^{d}}(\mathcal{S}_{{}^{\ast}\!(A+B)})\geq\mu_{\lbrack-H,H]^{d}}%
({}^{\ast}\!{A}+\mathbb{Z}^{d})>\alpha+\epsilon.
\]
By Proposition \ref{prop:syndeticimpliesresult} there must exist a standard
$m\in\mathbb{N}$ such that for all (standard) $k\in\mathbb{N}$
\[
\mu_{[-H,H]^{d}}\left(  \left\{  z\in[-H,H]^{d}:z+[-k,k]^{d}\subseteq\text{
}{}^{\ast}\!(A+B)+[-m,m]^{d}\right\}  \right)  \geq\alpha.
\]

By the nonstandard characterization of upper asymptotic density (Proposition
\ref{prop:nonstequivalents}) we obtain the desired result.
\end{proof}

The theorem above is, in general, the best possible, as is shown in the
example below in one dimension, with densities defined on $[1,n]$ rather than
$[-n,n]$.

\begin{example}
\label{ex:upper density theorem best possible}Let
\[
A=\bigcup_{n=1}^{\infty}[2^{n},2^{n}+2^{n-1}]\text{ and }B=\bigcup
_{n=1}^{\infty}[n!,n!+n].
\]
then
\[
\underline{\operatorname{d}}(A)=1/2=\underline{\operatorname{d}}%
(\{n:[n,n+k]\subseteq A+B+[0,m]\}))
\]
for all $m$ and $k$, and
\[
\overline{\operatorname{d}}(A)=2/3=\overline{\operatorname{d}}%
(\{n:[n,n+k]\subseteq A+B+[0,m]\}))
\]
for all $m$ and $k$ (note that these densities would be $1/4$ and $1/3$ if we
defined the densities on $[-n,n]$ as in our general definition).
\end{example}

We note that if $A$ is the set from the Example
\ref{ex:upper density theorem best possible}, and the set $C$ is defined to
equal $A\cap\lbrack(2n)!,(2n+1)!)$ on all $[(2n)!,(2n+1)!)$, and
$(2\mathbb{N)\cap}[(2n+1)!,(2n+2)!)$ on all $[(2n+1)!,(2n+2)!)$, then $C$ is
an example of a set with the property that $\lim_{i\rightarrow\infty
}(\underline{\operatorname{d}}(C^{[i]}))=\underline{\operatorname{d}}(C)>0$
but $C$ is not lower thick of level $\beta$ for any $\beta>0$.

In Example \ref{ex:upper density theorem best possible}\ where the results are
sharp we note that the densities are the same for $A$ as they are for every
$A^{[j]}$ for $j$ finite, and the conclusion holds with $m=0$. \ This suggests
the following slightly stronger version of the theorem above. \ The proof is
immediate from the proof of the theorem above.

\begin{corollary}
\label{Corollary: upper thick} Let $A$ and $B$ be subsets of $\mathbb{Z}^{d}$
with the property that $\overline{\operatorname{d}}(A)=\alpha>0$ and
$\operatorname{BD}(B)>0$. \ Let $\alpha^{\prime}=\lim_{i\rightarrow\infty
}(\overline{\operatorname{d}}(A_{[i]}))$. \ If $\alpha^{\prime}>\alpha$ then
$A+B$ is upper syndetic of level $r$ for any $r<\alpha^{\prime}$. \ If
$\alpha=\alpha^{\prime}$ then $A+B$ is upper thick of level $\alpha.$
\end{corollary}

Combining ideas from Example \ref{ex:we need the epsilon} and Example
\ref{ex:upper density theorem best possible} it is easy to see that Corollary
\ref{Corollary: upper thick} cannot be improved to allow $r$ to equal
$\alpha^{\prime}$. \ The set $A$ from Example \ref{ex:we need the epsilon} has
$\lim_{i\rightarrow\infty}(\overline{\operatorname{d}}(A_{[i]}))=1$, and if we
add that set to the set $B$ from Example
\ref{ex:upper density theorem best possible} then $A+B$ is not upper syndetic
of level 1.

\section{Lower syndeticity and
sumsets\label{Section: Lower syndeticity and sumsets}}

In this section we focus on how the previous theorem can be improved if the
set $A$ has the stronger property of positive lower density. \ In the proof of
Theorem \ref{upper density sum theorem} we used the fact that if $C\subseteq$
$\mathbb{Z}^{d}$ is such that for some $H\in$ $^{\ast}\mathbb{N}%
\left\backslash \mathbb{N}\right.  $
\[
\mu_{[-H,H]^{d}}\left(  \mathcal{S}_{^{\ast}C}\right)  >\alpha,
\]
then $C$ is upper syndetic of level $\alpha$. The analogous result for lower
density is far from true, as Example \ref{ex:optimal} shows (in one dimension).

\begin{example}
\label{ex:optimal} The set $C$ constructed below has the property that almost
all points in $^{\ast}C$ (on any infinite interval) are points of syndeticity
of $^{\ast}C$, and $\underline{\operatorname{d}}(C)=1/2$. \ However, for any
$m$
\[
\underline{\operatorname{d}}\{n\in\mathbb{N}:n+[-2m,2m]\subseteq A+[0,m]\}=0.
\]

Let $s_{i}$ be the sequence 1,2,1,2,3,1,2,3,4,1,2,3,4,5..., and let
$C\subseteq\mathbb{N}$ be such that:
\[
\text{on }[i!,(i+1)!),\text{ }n\in C\text{ iff }n\equiv\{0,1,..,s_{i}%
-1\}\operatorname{mod}2s_{i}.
\]
Thus, on $[i!,(i+1)!)$, $C$ consists of blocks of length $s_{i}$, with the
blocks alternating between being completely contained in $C$ and not
intersecting $C$. \ We note that for any given $m$, if we let $H=(I+1)!$ and
$s_{I}$ be such that $2m<s_{I}<3m$ then
\[
\mu_{[1,H]}\left(  \{n\in\text{ }^{\ast}\mathbb{N}:n+[-2m,2m]\subseteq\text{
}^{\ast}C+[0,m]\}\right)  =0.
\]
\ We also note that the only points in $^{\ast}C\left\backslash C\right.  $
that are not points of syndeticity are those that are within a standard
distance of an endpoint of one of the intervals of nonstandard length.
\end{example}

Example \ref{ex:optimal} shows that we cannot use the same proof technique
from Section \ref{Section: Upper syndeticity and sumsets} if we want to prove
an analogous result with a conclusion involving lower density. \ These
techniques do allow us to conclude strong upper syndeticity.

\begin{theorem}
\label{thm:lower density strongly upper syndetic theorem}Let $A$ and $B$ be
subsets of $\mathbb{Z}^{d}$ with the property that
$\underline{\operatorname{d}}(A)=\alpha>0$ and $\operatorname{BD}(B)>0$.
\ Then $A+B$ is strongly upper syndetic of level $\alpha$.
\end{theorem}

\begin{proof}
Let $S\subseteq\mathbb{N}$ be any sequence going to infinity. \ Let $H=s_{I}$,
where $I\in$ $^{\ast}\mathbb{N}\left\backslash \mathbb{N}\right.  $. \ Since
$\underline{\operatorname{d}}(A)=\alpha$ we know that $\mu_{[-H,H]^{d}}%
({}^{\ast}\!{A})\geq\alpha$. \ If
\[
\mu_{[-H,H]^{d}}({}^{\ast}\!{A}+\mathbb{Z}^{d})=\mu_{[-H,H]^{d}}(^{\ast}A),
\]
then by Proposition \ref{prop:unchanged measure proposition}, for each
standard $k\in\mathbb{N}$
\[
\mu_{[-H,H]^{d}}(\{x\in\text{ }{}^{\ast}\!{A}:x+[-k,k]^{d}\subseteq^{\ast
}A\})=\alpha,
\]
and we may let $m=0$. \ 

If
\[
\mu_{[-H,H]^{d}}({}^{\ast}\!{A}+\mathbb{Z}^{d})>\mu_{[-H,H]^{d}}({}^{\ast
}\!{A}),
\]
then by arguments identical to those used in the proof of Theorem
\ref{upper density sum theorem}, there must exist $m\in\mathbb{N}$ such that
for all $k\in\mathbb{N}$
\[
\mu_{[-H,H]^{d}}\left(  \left\{  z\in[-H,H]^{d}:z+[-k,k]^{d}\subseteq\text{
}{}^{\ast}\!(A+B)+[0,m]^{d}\right\}  \right)  \geq\alpha.
\]
The result now follows by transfer, since for this $m,$ any $\epsilon>0$ and
any $i,k\in\mathbb{N}$ there exists $j>i$ such that
\[
\frac{1}{(2s_{j}+1)^{d}}\left\vert \left\{  z\in\lbrack-s_{j},s_{j}%
]^{d}:z+[-k,k]^{d}\subseteq A+B+[0,m]^{d}\right\}  \right\vert \geq
\alpha-\epsilon.
\]

\end{proof}

In fact it is not true that the conclusion in the theorem above can be
improved to \textquotedblleft$A+B$ is lower syndetic of level $\alpha
$\textquotedblright\ (see Example \ref{ex:big example} below).
\ The\ following theorem is the strongest conclusion we can make involving
lower syndeticity.

\begin{theorem}
Let $A$ and $B$ be subsets of $\mathbb{Z}^{d}$ with the property that
$\underline{\operatorname{d}}(A)=\alpha>0$ and $\operatorname{BD}(B)>0$.
\ Then for any $\epsilon>0$, $A+B$ is lower syndetic of level $\alpha
-\epsilon$.
\end{theorem}

\begin{proof}
Without loss of generality, let $\epsilon<\alpha/2$. So we can assume that
$\alpha-\epsilon>\epsilon$. Choose $m\in\mathbb{N}$ sufficiently large so that
$\operatorname{BD}(B+[-m,m]^{d})>1-\epsilon$. \ Let $B_{m}=B+[-m,m]^{d}$. It
suffices to show that for any $H\in\,^{\ast}\mathbb{N}\smallsetminus
\mathbb{N}$ and any $k\in\mathbb{N}$,
\[
\mu_{[-H,H]^{d}}\left(  \left\{  x\in[-H,H]^{d}:x+[-k,k]^{d}\subseteq
\,{}^{\ast}\!(A+B)+[-m,m]^{d}\right\}  \right)  \geq\alpha-\epsilon.
\]

We first note that by a pigeonhole argument we can prove that if
$x,y,x+y\in[-H,H]^{d}$ are such that
\begin{equation}
\overline{\operatorname{d}}((^{\ast}\!A-x)\cap\mathbb{Z}^{d}%
)+\underline{\operatorname{d}}((^{\ast}\!B_{m}-y)\cap\mathbb{Z}^{d})>1,
\label{pigeonhole}%
\end{equation}%
\begin{equation}
\mbox{then }\,(x+y+\mathbb{Z}^{d})\cap[-H,H]^{d}\subseteq\,^{\ast}(A+B_{m}).
\label{pigeonhole2}%
\end{equation}

This is true because if $x$ and $y$ satisfy \ref{pigeonhole} then any
$x^{\prime}\in x+\mathbb{Z}^{d}$ and $y^{\prime}\in y+\mathbb{Z}^{d}$ also
satisfy this condition, so that ${}^{\ast}\!{A}-x^{\prime}$ and $y^{\prime}-$
$^{\ast}B_{m}$ must intersect.

Now for any $n\in\mathbb{N}$, let
\[
S_{n}=\left\{  x\in\,^{\ast}\!A\cap[-H,H]^{d}:\frac{|(x+[-n,n]^{d}%
)\cap\,^{\ast}\!A)|}{\left(  2n+1\right)  ^{d}}>\frac{\epsilon}{2}\right\}
\,
\]
and
\[
S=\bigcap_{N=1}^{\infty}\bigcup_{n=N}^{\infty}S_{n}.
\]
Similarly define
\[
T_{n}=\left\{  x\in\,^{\ast}\!A\cap[-H,H]^{d}:\frac{|(x+[-n,n]^{d}%
)\cap\,^{\ast}\!A)|}{\left(  2n+1\right)  ^{d}}<\frac{2\epsilon}{3}\right\}
\,
\]
and
\[
T=\bigcup_{N=1}^{\infty}\bigcap_{n=N}^{\infty}T_{n}.
\]

It is now easy to verify the following \textbf{Facts}:

\begin{enumerate}
\item \label{itm1: first} $\!$ ${}^{\ast}\!{A}\left\backslash S\right.
\subseteq T$.

\item \label{itm1: second} If $x\in\,[-H,H]^{d}$, then $x\in S$ implies
$\overline{\operatorname{d}}((^{\ast}\!A-x)\cap\mathbb{Z}^{d})\geqslant
\epsilon/2$ and $\overline{\operatorname{d}}((^{\ast}\!A-x)\cap\mathbb{Z}%
^{d})>\epsilon/2$ implies $x\in S$.

\item \label{itm1: third} If $x\in\,^{\ast}\!A\cap\,[-H,H]^{d}$, then $x\in T$
implies $\underline{\operatorname{d}}((^{\ast}\!A-x)\cap\mathbb{Z}%
^{d})\leqslant2\epsilon/3$ and $\overline{\operatorname{d}}((^{\ast}%
\!A-x)\cap\mathbb{Z}^{d})<2\epsilon/3$ implies $x\in T$.

\item \label{itm1: fourth} If $x,y\in\,^{\ast}\!A\cap[-H,H]^{d}$ and
$x-y\in\mathbb{Z}^{d}$, then $x\in S$ if and only if $y\in S$ and $x\in T$ if
and only if $y\in T$.
\end{enumerate}

We show that $\mu_{[-H,H]^{d}}(T)\leqslant2\epsilon/3$. \ Suppose that
$\mu_{[-H,H]^{d}}(T)=\gamma>2\epsilon/3$. \ By the Birkhoff Ergodic Theorem
the asymptotic density of $T-x$ exists for almost all $x$ (see
e.g.\ \cite{jinintro} pages 23 and 24 for more details on a similar argument
using this theorem). \ So there exists an $x\in T$ such that $d((T-x)\cap
\mathbb{Z}^{d})\geq\gamma$. \ By Fact \ref{itm1: fourth}, we have that
\[
(^{\ast}\!A-x)\cap\mathbb{Z}^{d}=(T-x)\cap\mathbb{Z}^{d}.
\]
Therefore, $\underline{\operatorname{d}}((^{\ast}\!A-x)\cap\mathbb{Z}%
)=\gamma>2\epsilon/3$, which contradicts that $x\in T$ by Fact
\ref{itm1: third}.

By Fact \ref{itm1: first} $\mu_{[-H,H]^{d}}(S)\geqslant\alpha-2\epsilon
/3>\alpha-\epsilon$. Let $t\in[-H,H]^{d}$ be such that $\left\Vert
t\right\Vert /H\approx0$ and $d(^{\ast}\!B_{m}\cap(t+\mathbb{Z}^{d}%
))>1-\epsilon$. \ By (\ref{pigeonhole})--(\ref{pigeonhole2}), we have that%

\begin{align*}
& ({}^{\ast}\!(A+B)+[-m,m]^{d})\cap[-H,H]^{d}\\
& =(^{\ast}\!A+\,^{\ast}\!B_{m})\cap[-H,H]^{d}\supseteq(t+S+\mathbb{Z}%
^{d})\cap[-H,H]^{d}.
\end{align*}
Consequently for any $k\in\mathbb{N}$, the measure of
\[
\left(  \{x\in[-H,H]^{d}:x+[0,k]^{d} \subseteq\,^{\ast}(A+B)+[-m,m]^{d}%
\}\right)
\]
is at least the measure of $S+\mathbb{Z}^{d}$, which is greater than or equal
to $\alpha-\epsilon$.

This completes the proof.
\end{proof}

Example \ref{ex:big example} shows (in one dimension) that we may not replace
$\alpha-\epsilon$ with $\alpha$ in the conclusion of the previous theorem.

\begin{example}
\label{ex:big example}Sets $A,B\subseteq\mathbb{N}$ can be constructed so that
they satisfy that $\underline{\operatorname{d}}(A)=1/2$, $\operatorname{BD}%
(B)\geqslant8/9$, and for any $m\in\mathbb{N}$ there exists $k\in\mathbb{N}$
such that
\[
\underline{\operatorname{d}}(\{x\in\mathbb{N}:x+[0,k]\subseteq
A+B+[0,m]\})<\frac{1}{2}.
\]

\end{example}

\begin{proof}
We construct $A$ first. Let $f(n,p)=10^{(n^{2}+p)^{2}}$. Notice that
$f(n,p)<f(n+1,0)$ for any $n$ and $p\leqslant n$, and
$f(N,p)/(f(N,p+1)-f(N,p))\approx0$ and $f(N,N)/(f(N+1,0)-f(N,N))\approx0$ for
any hyperfinite $N$ and $p\leqslant N$.

For each $p\in\mathbb{N}$ let $r_{p}=\frac{10^{p}-2}{2(10^{p}-1)}$. Notice
that $r_{p}<\frac{1}{2}$ when $p$ is finite and $r_{p}\rightarrow\frac{1}{2}$
as $p\rightarrow\infty$. The number $r_{p}$ satisfies that $r_{p}+\frac
{1}{10^{p}}(1-r_{p})=\frac{1}{2}$.

For each $n\in\mathbb{N}$ and each $p=1,2,\ldots n-1$, let
\[
C_{n,p}=C_{n,p}^{\prime}\cup C_{n,p}^{\prime\prime}%
\]%
\[
\mbox{where }\,C_{n,p}^{\prime}=[f(n,p),f(n,p)+\left\lfloor r_{p}%
(f(n,p+1)-f(n,p))\right\rfloor ]
\]%
\[
\mbox{and }\,C_{n,p}^{\prime\prime}=\left(  10^{p}\mathbb{N}\right)
\cap\lbrack f(n,p)+\left\lfloor r_{p}(f(n,p+1)-f(n,p))\right\rfloor
,f(n,p+1)],
\]
i.e.\ $C_{n,p}\subseteq\lbrack f(n,p),f(n,p+1)]$ is the union of an interval
of length 
\[
\left\lfloor r_{p}(f(n,p+1)-f(n,p))\right\rfloor 
\]
and an
arithmetic progression of difference $10^{p}$ and length
\[
\left\lfloor \frac{1}{10^{p}}(1-r_{p})(f(n,p+1)-f(n,p))\right\rfloor .
\]
Let $C_{n,n}=2[f(n,n),f(n+1,0)]$. Let
\[
A=\bigcup_{n=1}^{\infty}\bigcup_{p=0}^{n}C_{n,p}.
\]
Clearly, $\underline{\operatorname{d}}(A)=1/2$.

Now we construct $B$.

For each $p\in\mathbb{N}$ let $E_{p}=\bigcup_{k=1}^{\infty}\left(
k10^{2p}-[1,10^{p}]\right)  $, $D_{p}=\mathbb{N}\setminus\bigcup_{k=1}^p E_k$, and
$D=\bigcap_{p=1}^{\infty}D_{p}$. Let $F_{n}=D\cap\lbrack0,10^{2n}-1]$. We list
the following \textbf{Facts}:

\begin{enumerate}
\item $\!$ Every interval of length $10^{2p}$ contains a gap of $D$ with
length at least $10^{p}$.

\item Also $d(D)\geqslant1-\sum_{p=1}^{\infty}\frac{1}{10^{p}}=8/9$.

\item $F_{n}=D_{n}\cap\lbrack0,10^{2n}-1]$.

\item For any $p^{\prime}\leqslant p$
\[
k10^{2p}+D_{p^{\prime}}\cap\lbrack0,10^{2p}-1]=D_{p^{\prime}}\cap\lbrack
k10^{2p},(k+1)10^{2p}-1].
\]

\item \label{itm2: fifth} For any $p^{\prime}<p^{\prime\prime}\leqslant p$,
\[
k10^{2p}+10^{2p^{\prime\prime}}+F_{p^{\prime}}\subseteq10^{2p}\mathbb{N}%
+F_{p}.
\]

\item For any $n\geq p$,
\[
\lbrack k10^{2p},(k+1)10^{2p}-1]\cap D_{n}\subseteq k10^{2p}+F_{p}.
\]

\item $10^{2p}\mathbb{N}+D=10^{2p}\mathbb{N}+F_{p}$.
\end{enumerate}

Let
\[
B=\bigcup_{n=2}^{\infty}(f(n,0)+F_{n}).
\]

Clearly, $B\!D(B)\geqslant d(D)\geqslant8/9$. For each hyperfinite integer $N$
and $1\leqslant p<N$, let $u=\max\,^{\ast}\!B\cap\lbrack0,f(N,p)]$. We have
that $u/f(N,p)\approx0$. The set $B$ is a union of $F_{n}$'s translated by
rapidly increasing powers of $10$. It is important to observe that
\[
10^{2n}\mathbb{N}+B\subseteq10^{2n}\mathbb{N}+F_{n}.
\]

This is true because by Fact \ref{itm2: fifth}, we have that

\begin{itemize}
\item if $f(n^{\prime},0)\geq10^{n}$, then
\[
k10^{2n}+f(n^{\prime},0)+F_{n^{\prime}}=k^{\prime}10^{2n}+F_{n^{\prime}%
}\subseteq10^{2n}\mathbb{N}+F_{n};
\]

\item if $f(n^{\prime},0)=10^{(n^{\prime})^{4}}<10^{2n}$, then
\[
k10^{2n}+f(n^{\prime},0)+F_{n^{\prime}}=k^{\prime}10^{(n^{\prime})^{4}%
}+F_{n^{\prime}}=k^{\prime\prime}10^{2(n^{\prime}+1)}+F_{n^{\prime}}%
\subseteq10^{2n}\mathbb{N}+F_{n}.
\]

\end{itemize}

Now we show that the sets $A$ and $B$ are what we want.

Given any $m\in\mathbb{N}$, choose a $p\in\mathbb{N}$ sufficiently large so
that $10^{p}>2m$. Let $H=f(N,2p+1)$ and $u=\max B\cap\lbrack1,H]$. Then 
\[
(f(N,2p+1)-f(N,2p))/H\approx1
\]
and
$u/H\approx0$. Therefore
\[
(^{\ast}\!A+\,^{\ast}\!B+[0,m])\cap\lbrack0,f(N,2p+1)]
\]
is contained in 
\begin{align*}
&  \lbrack0,f(N,2p)+\left\lfloor r_{2p}
(f(N,2p+1)-f(N,2p))\right\rfloor +u+m] \\
& \cup (C_{N,2p}^{\prime\prime}+\,^{\ast }\!B\cap\lbrack0,u]+[0,m]).
\end{align*}

Note that $(C_{N,2p}^{\prime\prime}+\,^{\ast}\!B\cap\lbrack0,u])\cap
\lbrack1,H]\subseteq(C_{N,2p}^{\prime\prime}+F_{2p})\cap\lbrack1,H]$ and every
interval of length $10^{2p}$ contains a gap of length $10^{p}$ in
$(C_{N,2p}^{\prime\prime}+F_{2p})\cap\lbrack1,H]$. Since $10^{p}>2m$, every
interval of length $10^{2p}$ in 
\[
[f(N,2p)+\left\lfloor r_{2p}(f(N,2p+1)-f(N,2p))\right\rfloor +u+m,f(N,2p+1)]
\]
 is not entirely in ${}^{\ast}\!(A+B)+[0,m]$. So we can choose $k=10^{2p}$ so that
\begin{align*}
\lefteqn{\{x\in\lbrack1,H]:x+[0,k]\subseteq\,{}^{\ast}\!(A+B)+[0,m]\}}\\
&  \subseteq\lbrack1,f(N,2p)+\left\lfloor r_{2p}%
(f(N,2p+1)-f(N,2p))\right\rfloor +u].
\end{align*}

Hence
\begin{align*}
\lefteqn{\mu_{\lbrack1,H]}(\{x\in\lbrack1,H]:x+[0,k]\subseteq\,{}^{\ast
}\!(A+B)+[0,m]\})}\\
&  \approx\frac{1}{H}(f(N,2p)+\left\lfloor r_{2p}%
(f(N,2p+1)-f(N,2p))+u)\right\rfloor \approx r_{2p}<\frac{1}{2}.
\end{align*}

\end{proof}

\section{Syndeticity for the sum of two sets of positive lower
density\label{Section: Syndeticity for the sum of two sets of positive lower density}%
}

In this section we focus only on the dimension 1 case, where the results from
Section \ref{Section: Lower syndeticity and sumsets} can be improved under the
assumption that both sets have positive lower density. \ The results use
Mann's Theorem about the additivity of Schnirelmann density \cite{Mann} and
thus do not generalize to $n$ dimensions in a straightforward way. \ For the
remainder of the section the dimension is 1 and the density functions are
defined on intervals of natural numbers starting at 1, as in the classical setting.

Mann's theorem asserts that if $A$ and $B$ are subsets of $\mathbb{N}$ such
that $\sigma(A)=\alpha$ and $\sigma(B)=\beta$, then \
\[
\sigma((A\cup\left\{  0\right\}  )+(B\cup\left\{  0\right\}  ))\geq
\min\{\alpha+\beta,1\}.\
\]
This guarantees that for any $n$ \
\[
\frac{\left\vert \left(  (A\cup\left\{  0\right\}  )+(B\cup\left\{  0\right\}
)\right)  \cap\lbrack1,n]\right\vert }{n}\geq\min\{\alpha+\beta,1\}.
\]
So the result can at once be thought of as pertaining to either infinite sets
or finite sets of natural numbers up to some $n$.

We first need the proposition below.

\begin{proposition}
\label{prop:existence of Schnirelmann sets}Let $A\subseteq\mathbb{N}$ be such
that $\underline{\operatorname{d}}(A)=\alpha>0.$ \ Then for any $H\in$
$^{\ast}\mathbb{N}\left\backslash \mathbb{N}\right.  $ and any $\epsilon>0$
there exists an internal $E\subseteq$ ${}^{\ast}\!{A}\cap\mathcal{D}_{{}%
^{\ast}\!{A}}\cap[1,H]$ such that $\sigma(E-e)\geq r-\epsilon$ for some $e\in$
$[1,H]$, with $e/H<\epsilon$.
\end{proposition}

Here by $\sigma(E-e)$ we mean the\ Schnirelmann density of the internal set
$E-e$ on $[1,H-e]$, i.e.\ $\inf_{h\in H-e}\frac{\left\vert (E-e)\cap
\lbrack1,h]\right\vert }{h}.$

\begin{proof}
We will write $D_{A}$ for ${}^{\ast}\!{A}\cap\mathcal{D}_{{}^{\ast}\!{A}}.$
Since $D_{A}$ is Loeb measurable on any interval, and Loeb measurable sets are
approximable from below by internal sets, for each $n\in\mathbb{N}$ there
exists an internal set $E_{n}$ such that \
\[
E_{n}\subseteq D_{A}\cap\lbrack(H/2^{n}),(H/2^{n-1})]
\]
and
\[
\mu_{[1,H]}(E_{n})>(1-\epsilon/4)\mu_{[1,H]}(D_{A}\cap\lbrack(H/2^{n}%
),(H/2^{n-1})]).
\]

Let $m\in\mathbb{N}$ be such that \
\[
1/2^{m}<\epsilon\text{ and let }E=\bigcup_{n=1}^{2m}(E_{n}).
\]
\ Then $E$ is internal and $E\subseteq D_{A}\cap[1,H]=$ $^{\ast}%
A\cap\mathcal{D}_{{}^{\ast}\!{A}}$\ $\cap[1,H].$

By Theorem \ref{thm:LDTtheorem1} and Proposition \ref{prop:nonstequivalents}
\[
\mu_{\lbrack1,H]}(D_{A}\cap\lbrack1,K])\geq\alpha(K/H)\text{ for all
}\mathbb{N<}K<H.
\]

Note that for any $n<m$, if $x\in\lbrack(H/2^{n}),(H/2^{n-1})]$ then
\begin{align*}
\mu_{[1,H]}(E\cap\lbrack(H/2^{n}),x])  &  \geq\mu_{_{[1,H]}}(D_{A}\cap
\lbrack(H/2^{n}),x])-(\epsilon/4)(1/2^{n-1}-1/2^{n})\\
&  \geq\mu_{_{[1,H]}}(D_{A}\cap\lbrack(H/2^{n}),x])-(\epsilon/4)(x/H),
\end{align*}
and
\[
\mu_{[1,H]}(E\cap\lbrack1,(H/2^{n})])>(1-\epsilon/4)\mu_{\lbrack1,H]}%
(D_{A}\cap\lbrack1,(H/2^{n})])-1/2^{2m}.
\]

We now have: \
\begin{align*}
&  \mu_{[1,H]}(E\cap\lbrack1,x])\\
=  & \mu_{[1,H]}(E\cap\lbrack1,(H/2^{n})])+\mu_{[1,H]}(E\cap\lbrack
(H/2^{n}),x])\\
>  & (1-\epsilon/4)\mu_{[1,H]}(D_{A}\cap\lbrack1,(H/2^{n})])-1/2^{2m}+\\
&  +\mu_{_{[1,H]}}(D_{A}\cap\lbrack(H/2^{n}),x])-(\epsilon/4)(x/H)\\
\geq & \mu_{[1,H]}(D_{A}\cap\lbrack1,x])-\epsilon/4(\mu_{_{[1,H]}}(D_{A}%
\cap\lbrack(H/2^{n}),x])+x/H)-\epsilon/2^{m}\\
\geq & \alpha x/H-\epsilon/4(x/H+x/H)-(\epsilon/2)x/H\\
\geq & \alpha x/H-\epsilon x/H.
\end{align*}

This means that the largest element $u$ in $[1,H]$ such that
\[
\left\vert E\cap\lbrack1,u]\right\vert <(\alpha-\epsilon)u
\]

is less than $(1/2^{m})H$. \ Let $e=u+1$. \ We note that $e$ must be an
element of $E$, and that for all $e<x<H$
\[
\left\vert E\cap\lbrack e+1,x]\right\vert \geq(\alpha-\epsilon)(x-e)
\]

by the maximality of $u$. \ Thus $\sigma(E-e)\geq\alpha-\epsilon$ on
$[1,H-e]$, and all the statements in the conclusion are satisfied.
\end{proof}

\begin{theorem}
\label{thm:lowerdensitysumresult} Let $A$ and $B$ be subsets of $\mathbb{N}$
with the property that $\underline{\operatorname{d}}(A)=\alpha>0$, and
$\underline{\operatorname{d}}(B)=\beta>0$. Then $A+B$ is strongly upper
syndetic of level $\min\{\alpha+\beta,1\}$. \ 
\end{theorem}

\begin{proof}
If $\alpha+\beta>1$ then $A+B$ contains all but finitely many positive
integers, hence the conclusion holds trivially with $m=0$. \ So, we suppose
that $\alpha+\beta\leq1$.

Let $S\subseteq\mathbb{N}$ be any sequence going to infinity. \ Let $H=s_{I}$,
where $I\in$ $^{\ast}\mathbb{N}\left\backslash \mathbb{N}\right.  $. \ 

By transfer (as in the proof of Theorem
\ref{thm:lower density strongly upper syndetic theorem}) it suffices to show
that there exists $m\in\mathbb{N}$ such that for all $k\in\mathbb{N}$
\[
\mu_{[1,H]}\left(  \left\{  z\in[1,H]:z+[-k,k]\subseteq\text{ }{}^{\ast
}\!(A+B)+[-m,m]\right\}  \right)  \geq\alpha+\beta.
\]

By Proposition \ref{prop:existence of Schnirelmann sets}, for each
$n\in\mathbb{N}$ there exists an internal set $E_{A,n}$ and $a_{n}\in E_{A,n}$
such that

\begin{itemize}
\item $\sigma(E_{A,n}-a_{n})\geq\alpha-1/n$ on $[1,H-a_{n}]$,

\item $E_{A,n}\subseteq\text{ }{}^{\ast}\!{A}\cap\mathcal{D}_{{}^{\ast}\!{A}%
}\cap[1,H]$, and

\item $a_{n}/H<1/n$.
\end{itemize}

Similarly, for each $n\in\mathbb{N}$ there exists $E_{B,n}$ and $b_{n}\in
E_{B,n}$ such that

\begin{itemize}
\item $\sigma(E_{B,n}-b_{n})\geq\beta-1/n$ on $[1,H-b_{n}]$,

\item $E_{B,n}\subseteq{}^{\ast}B\cap\mathcal{D}_{^{\ast}B}\cap[1,H]$, and

\item $b_{n}/H<1/n$.
\end{itemize}

By Mann's Theorem,
\[
\sigma(E_{A,n}-a_{n}+E_{B,n}-b_{n})\geq\alpha+\beta-2/n\text{ on }%
[1,H-(a_{n}+b_{n})],
\]
i.e.\
\[
\frac{\left\vert (E_{A,n}-a_{n}+E_{B,n}-b_{n})\cap\text{ }[1,H-(a_{n}%
+b_{n})]\right\vert }{H-(a_{n}+b_{n})}\geq\alpha+\beta-2/n\text{.}%
\]
This implies that
\[
\mu_{\lbrack1,H]}\left(  E_{A,n}-a_{n}+E_{B,n}-b_{n})\cap\lbrack
1,H-(a_{n}+b_{n})]\right)  \geq(\alpha+\beta-2/n)(1-2/n).
\]

Thus
\[
\mu_{[1,H]}(E_{A,n}+E_{B,n})\geq(\alpha+\beta-2/n)(1-2/n) \geq\alpha
+\beta-4/n.
\]

Since each $E_{A,n}$ is in ${}^{\ast}\!{A}\cap\mathcal{D}_{{}^{\ast}\!{A}}$
and each $E_{B,n}$ is in ${}^{\ast}B\cap\mathcal{D}_{{}^{\ast}B}$, by Theorem
\ref{thm:density plus density is syndetic} we know that every $E_{A,n}%
+E_{B,n}$ is contained in ${}^{\ast}(A+B)\cap\mathcal{S}_{{}^{\ast}\!(A+B)}$,
so that
\[
\mu_{[1,H]}({}^{\ast}(A+B)\cap\mathcal{S}_{{}^{\ast}\!(A+B)})\geq\alpha
+\beta.
\]

Now, if
\[
\mu_{[1,H]}(\mathcal{S}_{{}^{\ast}\!(A+B)})>\alpha+\beta
\]
then the result follows by Proposition \ref{prop:syndeticimpliesresult}.
\ \ If, on the other hand,
\[
\mu_{[1,H]}(\mathcal{S}_{{}^{\ast}\!(A+B)})=\alpha+\beta
\]
then, since $\mathcal{S}_{{}^{\ast}\!(A+B)}=\mathcal{S}_{{}^{\ast}%
\!(A+B)}+\mathbb{Z}$, it must be that
\begin{align*}
\alpha+\beta &  \leq\mu_{[1,H]}({}^{\ast}(A+B)\cap\mathcal{S}_{{}^{\ast
}\!(A+B)})\leq\mu_{[1,H]}(({}^{\ast}(A+B)\cap\mathcal{S}_{{}^{\ast}%
\!(A+B)})+\mathbb{Z})\\
&  \leq\mu_{[1,H]}(\mathcal{S}_{{}^{\ast}\!(A+B)}+\mathbb{Z})=\mu
_{[1,H]}(\mathcal{S}_{{}^{\ast}\!(A+B)})=\alpha+\beta
\end{align*}
Thus, the set ${}^{\ast}(A+B)\cap\mathcal{S}_{{}^{\ast}\!(A+B)}$ satisfies the
hypotheses of the set $S$ in Proposition
\ref{prop:unchanged measure proposition} (in one dimension).\ \ This implies
that for each standard $k\in\mathbb{N}$
\[
\mu_{[1,H]}(\{x\in\text{ }{}^{\ast}(A+B):x+[-k,k]\subseteq{}^{\ast
}(A+B)\}=\alpha+\beta,
\]
and the result follows with $m=0$.
\end{proof}

\textbf{Question: }Under the same hypotheses as in the theorem above, can we
conclude that for any $\epsilon>0$ the sumset $A+B$ is lower syndetic of level
$\min\{\alpha+\beta-\epsilon,1\}$?

Currently the strongest conclusion that can be made involving lower density is
the result below.

\begin{theorem}
Let $A$ and $B$ be subsets of $\mathbb{N}$ with the property that
$\underline{\operatorname{d}}(A)=\alpha>0$, and $\underline{\operatorname{d}%
}(B)=\beta>0$. \ Then for any\ $\epsilon>0$ and any increasing function
$f:\mathbb{N\rightarrow N}$, there exists $m_{f}\in\mathbb{N}$ such that
\[
\underline{\operatorname{d}}(\{n\in\mathbb{N}:\text{ }\exists\text{ }%
m<m_{f}\text{, }n+[-f(m),f(m)]\subseteq A+B+[-m,m]\})
\]
is at least $\min\{\alpha+\beta-\epsilon,1\}$.
\end{theorem}

We note that here $m_{f}$ depends only on the function, but that $m$ may
depend on $n$.

\begin{proof}
As before, if $\alpha+\beta>1$ the result is immediate, so we assume that
$\alpha+\beta\leq1$ and suppose, for the sake of contradiction, that for some
$\epsilon>0$ no such $m_{f}$ exists. \ Then there exists $r<\alpha+\beta$ such
that for all $m_{0}\in\mathbb{N}$ there exist arbitrarily large $n\in
\mathbb{N}$ such that for all $m<m_{0}$ \
\[
\left\vert \{z\in\mathbb{[}1,n]:z+[-f(m),f(m)]\subseteq
A+B+[-m,m]\}\right\vert <rn.
\]
By overspill there exist $M,H\in$ $^{\ast}\mathbb{N}\left\backslash
\mathbb{N}\right.  $ such that for all $m<M$ \
\[
\left\vert \{z\in\mathbb{[}1,H]:z+[-f(m),f(m)]\subseteq{}^{\ast}%
(A+B)+[-m,m]\}\right\vert <rH,
\]
so that
\[
\mu_{[1,H]}\left(  \left\{  z\in\mathbb{[}1,H]:z+[-f(m),f(m)]\subseteq\text{
}{}^{\ast}(A+B)+[-m,m]\right\}  \right)  \leq r.
\]
But, as in the proof of the previous theorem, we know that for any fixed
$H\in$ $^{\ast}\mathbb{N}\left\backslash \mathbb{N}\right.  $ there exists
$m\in\mathbb{N}$ such that for all $k\in\mathbb{N}$
\[
\mu_{[1,H]}(\{z\in[1,H]:z+[-k,k]\subseteq{}^{\ast}(A+B)+[-m,m]\})\geq
\alpha+\beta,
\]
and this contradiction completes the proof.
\end{proof}

\section{A Lebesgue density theorem for nonstandard
cuts\label{Section: A Lebesgue density theorem for nonstandard cuts}}

The classical Lebesgue Density Theorem for $\mathbb{R}^{d}$ says that if $E$
is a Lebesgue measurable set in $\mathbb{R}^{d}$ then almost every point in
$E$ is a point of density of $E$, i.e.\ almost every $x\in E$ has the property
that
\[
\lim_{\epsilon\rightarrow0}\frac{\lambda(E\cap(x+(-\epsilon,\epsilon)^{d}%
))}{(2\epsilon)^{d}}=1.
\]
The goal of this section is to prove an analogue of the Lebesgue Density
Theorem for measures induced by arbitrary cuts in $^{\ast}\mathbb{N}$. \ Let
$H\in$ $^{\ast}\mathbb{N}\left\backslash \mathbb{N}\right.  $. \ A \emph{cut}
$U$ in $[1,H]$ is an initial segment of $[1,H]$ that is closed under addition.
\ Cuts in this context were introduced in \cite{KL}, and some of the
topological properties of the quotient space $[1,H]$ under the equivalence
relation $x\equiv y$ iff $\left\vert x-y\right\vert \in U$ were explored. \ 

For a given cut $U$ in $[1,H]$, we let $\mathcal{U=}(-U)\cup\{0\}\cup(U)$. \ A
$U-$\emph{monad of }$[-H,H]^{d}$ is a set of the form $x+\mathcal{U}^{d}$,
where $x\in[-H,H]^{d}$ and $x+\mathcal{U}^{d}\subseteq[-H,H]^{d}$. \ The main
result in this section is really about the behavior of Loeb measure on the
space of monads of various cuts, i.e.\ the quotient space under the projection
that sends $x$ to $x+\mathcal{U}^{d}$. \ For any $\mathbb{N<}K\leq H$ there is
a natural cut of all elements infinitesimal to $K$, given by $U_{K}%
=\bigcap_{i=1}^{\infty}[1,K/i]$. \ Loeb measure on the quotient space of
$[-K,K]^{d}$ for the cut $U_{K}$ is isomorphic to Lebesgue Measure on
$[-1,1]^{d}$ via the measure-preserving mapping that sends $x+\mathcal{U}^{d}$
to $\operatorname{st}(x/K)$. \ So, the fact that the Lebesgue Density Theorem
holds for such cuts is immediate from the fact that the result holds for
Lebesgue measure. \ Previous standard results were obtained by using the
density theorem in the space of monads of such $U_{K}$ in \cite{leth1},
\cite{leth2} and \cite{leth3}. \ In this section we show that there is an
analogous density theorem for every cut and in every finite dimension. \ The
standard results in this paper are based on the density theorem in the case
where that $U=\mathbb{N}$.

We begin with a standard combinatorial lemma.

\begin{lemma}
\label{lemma:approximate tilings dimension N}Suppose that $m\in\mathbb{N}$ and
$\left(  T_{i}\right)  _{i<n}$ is a collection of subsets of a finite set $X$
such that for every $x\in X$%
\[
1\leq\sum_{i<n}\chi_{T_{i}}(x)\leq m
\]
where $\chi_{T_{i}}$ denotes the characteristic function of $T_{i}$. If
$t\in\left(  0,1\right)  $ and $E\subset X$ is such that%
\[
\frac{\left\vert T_{i}\cap E\right\vert }{\left\vert T_{i}\right\vert }\leq t
\]
for every $i<n$, then%
\[
\frac{\left\vert E\right\vert }{\left\vert X\right\vert }\leq\frac
{mt}{1+\left(  m-1\right)  t}{.}%
\]

\end{lemma}

\begin{proof}
We may assume that each $x\in E$ is in only one of the $T_{i}$ and that each
$x\in X\left\backslash E\right.  $ is in $m$ of the $T_{i}$ since removing
elements of $E$ from all but one of the $T_{i}$ or adding elements of
$X\left\backslash E\right.  $ to any of the $T_{i}$ (if that element is in
fewer than $m$ of them) maintains the hypotheses without changing the
conclusion. \ Then
\[
\left\vert E\right\vert \leq t\sum_{i<n}\left\vert T_{i}\right\vert
=t(m(\left\vert X\right\vert -\left\vert E\right\vert )+\left\vert
E\right\vert )=t(m\left\vert X\right\vert -(m-1)\left\vert E\right\vert )
\]

so that
\[
\left\vert E\right\vert (1+(m-1)t)\leq tm\left\vert X\right\vert
\]

which yields the desired result.
\end{proof}

If $E$ is an internal subset of $^{\ast}\mathbb{Z}^{d}$ and $x\in{}^{\ast
}\mathbb{Z}^{d}$ define%
\[
d_{E}(x):=\liminf_{\nu>U}\mu_{x+[-\nu,\nu]^{d}}\left(  \left(  E+\mathcal{U}%
^{d}\right)  \cap\left(  x+\left[  -\nu,\nu\right]  ^{d}\right)  \right)
\text{, }%
\]
where $\liminf_{\nu>U}$ means $\sup_{\xi>U}\inf_{U<\nu<\xi}$. \ Observe that
if $x\in$ $^{\ast}\mathbb{Z}^{d}$ and $y\in U^{d}$ then%
\[
d_{E}\left(  x+y\right)  =d_{E}\left(  x-y\right)  =d_{E}(x)\text{.}%
\]

The proof of the next theorem is based on the proof of the Lebesgue Density
Theorem given in \cite{faure}.

\begin{theorem}
\label{thm:LDTtheorem2}Let $H\in$ $^{\ast}\mathbb{N}\left\backslash
\mathbb{N}\right.  $ and $E$ be an internal subset of $[-H,H]^{d}$. Then%
\[
\ \mu_{[-H,H]^{d}}\left(  \left\{  x\in E+\mathcal{U}^{d}:d_{E}(x)<1\right\}
\right)  =0.
\]

\end{theorem}

\begin{proof}
We will write simply $\mu$ for\ $\mu_{[-H,H]^{d}}$. \ Until we are able to
show that the outer measure of $\left\{  x\in E+\mathcal{U}^{d}:d_{E}%
(x)<1\right\}  $ is 0, it is not clear that the set is measurable. \ To show
this, we fix $t\in\left(  0,1\right)  $, and prove that the set%
\[
R=\left\{  x\in E+\mathcal{U}^{d}:d_{E}(x)<t\right\}
\]
has outer measure $0$. \ For any $\epsilon>0$ \ we may pick an internal subset
$D$ of $[-H,H]^{d}$ containing $R$ such that%
\[
\mu\left(  D\right)  \leq\mu^{\ast}\left(  R\right)  +\epsilon{.}%
\]

We will show that\
\[
\mu\left(  D\right)  \leq\frac{\epsilon}{1-\frac{4^{d}t}{(4^{d}-1)t+1}},
\]
which can be made arbitrarily small by making $\epsilon$ small. \ This yields
the desired result since $R\subseteq D$.

Define%
\[
R_{+}=\left\{  x\in E:\exists z\in{}[1,H],x+\left[  -z,z\right]  ^{d}\subset
D\text{ and }\frac{\left\vert E\cap(x+\left[  -z,z\right]  ^{d})\right\vert
}{\left(  2z+1\right)  ^{2}}<t\right\}
\]
and observe that $R_{+}$ is an internal subset of $E$ containing $R\cap E$. We
can now cover every point $x$ in $D$ by\ open hypercubes of the form
$(x+(-y,y)^{d}),$ such that
\[
\frac{\left\vert R_{+}\cap(x+(-y,y)^{d})\right\vert }{\left(  2y+1\right)
^{2}}\leq t,
\]
by letting $y=z+\frac{1}{2}$ if $x\in R_{+}$ and $y=\frac{1}{2}$ if $x\in
D\left\backslash R_{+}\right.  $. \ By the Besicovitch Covering Theorem (see,
for example, \cite[page 483]{Jones}) there exists a hyperfinite sequence
$\left(  S_{i}\right)  _{i\in I}$ of these hypercubes such that for every
$x\in D$%
\[
1\leq\sum_{i\in I}\chi_{S_{i}}(x)\leq4^{d}.
\]

An application of Lemma \ref{lemma:approximate tilings dimension N} shows that%
\[
\frac{\left\vert R_{+}\right\vert }{\left\vert D\right\vert }\leq\frac{4^{d}%
t}{(4^{d}-1)t+1}.
\]
It follows that%
\begin{align*}
\mu\left(  D\right)   &  \leq\mu\left(  R_{+}\right)  +\epsilon\\
&  \leq\frac{4^{d}t}{(4^{d}-1)t+1}\mu\left(  D\right)  +\epsilon
\end{align*}
and hence%
\[
\mu\left(  D\right)  \leq\frac{\epsilon}{1-\frac{4^{d}t}{(4^{d}-1)t+1}}%
\]
as desired, showing that the outer measure is 0.

By the completeness of the Loeb measure we obtain the desired result.
\end{proof}

Suppose that $U$ is a cut in $[1,H]$.\ We say that a subset $S$ of
$[-H,H]^{d}$ is\newline$\emph{U}$\emph{-hereditarily measurable} iff for every
$x\in[-H,H]^{d}$ and every $U\mathbb{<}\nu<H$,
\[
(S+\mathcal{U}^{d})\cap(x+[-\nu,\nu]^{d})\text{ is }\mu_{x+[-\nu,\nu]^{d}%
}-\text{measurable.}%
\]
For $x\in[-H,H]^{d}$ and $S\subseteq[-H,H]^{d}$ $U-$hereditarily measurable we
define:%
\[
d_{S}^{U}(x)=\liminf_{\nu>U}\mu_{x+[-\nu,\nu]^{d}}\left(  (S+\mathcal{U}%
^{d})\cap(x+\left[  -\nu,\nu\right]  ^{d})\right)  \text{.}%
\]
\ We note that since $S$ is hereditarily measurable $d_{S}^{U}$ is
well-defined. \ If $U=\mathbb{N}$ and $S$ is internal this definition agrees
with the definition given in Section \ref{Section: Points of Density}.
\ Equivalently, we adopt the convention that if $U=\mathbb{N}$ we simply write
$d_{S}(x)$ for $d_{S}^{\mathbb{N}}(x).$ \ As in Section
\ref{Section: Points of Density} we say that $x$ is a point of density of $S$
iff $d_{S}^{U}(x)=1,$ and we write $\mathcal{D}_{S}^{U}$ for the set of all
points of density of $S$ with respect to the cut $U$.

We say that a cut $U$ has $\emph{countable}$ $\emph{cofinality}$ iff there
exists an increasing sequence $x_{n}\in$ $^{\ast}\mathbb{N}$ such that
$\bigcup_{n\in\mathbb{N}}[1,x_{n}]=U$, and that $U$ has $\emph{countable}$
$\emph{coinitiality}$ iff there exists a decreasing sequence $x_{n}\in$
$^{\ast}\mathbb{N}$ such that
\[
\bigcap_{n\in\mathbb{N}}[1,x_{n}]=U.\
\]

\begin{proposition}
For any cut $U$ in $[1,H]$, every internal set contained in $[-H,H]^{d}$ is
$U-$hereditarily measurable. \ 
\end{proposition}

\begin{proof}
Since any internal set intersected with any $x+[-\nu,\nu]^{d}$ is internal, it
suffices to show that for any internal set $E$, $E+\mathcal{U}^{d}$ is
$\mu_{x+[-\nu,\nu]^{d}}-$measurable. \ We will simply write $\mu$ for
$\mu_{x+[-\nu,\nu]^{d}}$. \ If $U$ has countable cofinality then
$E+\mathcal{U}^{d}$ is a countable union of internal sets of the form
$E\pm\lbrack1,x_{n}]^{d}$ and so is measurable. \ If $U$ has countable
coinitiality then $E+\mathcal{U}^{d}$ is a countable intersection of internal
sets of the form $E\pm\lbrack1,x_{n}]^{d}$ and so is measurable. \ \ So, we
assume that $U$ has neither countable coinitiality nor countable cofinality.
\ Let
\[
\gamma=\inf\left\{  \mu\left(  E+[-K,K]^{d}\right)  :K>U\right\}
\]
and
\[
\delta=\sup\left\{  \mu\left(  E+[-K,K]^{d}\right)  :K<U\right\}  ,
\]
It suffices to show that $\gamma=\delta$. \ Assume the contrary, that
$\gamma>\delta$. \ 

Let $K_{n}>U$ be decreasing such that
\[
\mu(E+[-K_{n},K_{n}]^{d})<\gamma+1/n
\]
for all $n\in\mathbb{N}$. \ Since the coinitiality of $U$ is uncountable,
there exists $K^{\prime}>U$ such that
\[
\text{for any }K\leq K^{\prime}\text{ and }K>U\text{, }\mu(E+[-K,K]^{d}%
)=\gamma.
\]
Symmetrically, we can find a $K^{\prime\prime}<U$ such that
\[
\text{for any }K\geq K^{\prime\prime}\text{ and }K<U\text{, }\mu
(E+[-K,K]^{d})=\delta.
\]

Let $\eta=\frac{1}{2}(\gamma+\delta)$, and let
\[
X=\left\{  K\in\lbrack K^{\prime\prime},K^{\prime}]:\left\vert E+[-K,K]^{d}%
\right\vert /(2\nu+1)^{d}\leq\eta\right\}  .
\]
Then $X$ is internal and $U\cap\lbrack K^{\prime\prime},K^{\prime}]\subseteq
X$. \ So, $X\cap\left(  \lbrack K^{\prime\prime},K^{\prime}]\smallsetminus
U\right)  $ is nonempty. \ Let $K\leq K^{\prime}$ and $K>U$ be such that
\[
\mu\left(  E+[-K,K]^{d}\right)  \approx\left\vert E+[-K,K]^{d}\right\vert
/(2\nu+1)^{d}\leq\eta.
\]
This contradicts the fact that $\gamma>\eta$.
\end{proof}

\begin{proposition}
Let $H\in$ $^{\ast}\mathbb{N}\left\backslash \mathbb{N}\right.  $, $U$ be a
cut in $[1,H]$, and $S\subseteq[-H,H]^{d}$ be hereditarily measurable. \ Then
the set%
\[
\left\{  x\in S+\mathcal{U}^{d}:d_{S}(x)<1\right\}
\]
has Loeb measure zero relative to $[-H,H]^{d}$.
\end{proposition}

\begin{proof}
Fix an $\epsilon>0$. \ Since $S+\mathcal{U}^{d}$ is measurable there exists
$E\subseteq S+\mathcal{U}^{d}$ such that $E$ is internal and $\mu_{[-H,H]^{d}%
}(S\left\backslash E\right.  )<\epsilon$. \ Then
\[
\left\{  x\in S+\mathcal{U}^{d}:d_{S}(x)<1\right\}  \subseteq\left\{  x\in
E+\mathcal{U}^{d}:d_{E}(x)<1\right\}  \cup(S+\mathcal{U}^{d}\left\backslash
E\right.  ).
\]

It follows that the outer measure of
\[
\left\{  x\in S+\mathcal{U}^{d}:d_{S}(x)<1\right\}
\]
is at most $\epsilon$. Since $\epsilon$ is arbitrary, the outer measure is 0,
and the result follows by the completeness of the Loeb measure.
\end{proof}

Corollary \ref{Corollary: LDT} is a generalization of Theorem
\ref{thm:LDTtheorem1}.

\begin{corollary}
\label{Corollary: LDT} If $E$ is an internal subset of $[-H,H]^{d}$ and $U$ is
a cut in $[1,H]$ then $\mathcal{D}_{E}^{U}$ is $\mu_{[-H,H]^{d}}$-measurable,
and $\mu_{[-H,H]^{d}}(\mathcal{D}_{E}^{U})=\mu_{[-H,H]^{d}}(E+\mathcal{U}%
^{d})$.
\end{corollary}

\begin{proof}
$\mathcal{D}_{E}^{U}=(E+\mathcal{U}^{d})\left\backslash \left\{  x\in
E+\mathcal{U}^{d}:d_{E}(x)<1\right\}  \right.  $, and the conclusion follows.
\end{proof}

It would be interesting to know if the results of Section
\ref{Section: Upper syndeticity and sumsets} and Section
\ref{Section: Lower syndeticity and sumsets} generalize to more general
amenable groups.

It would also be interesting to know if the density theorem in the space of
monads of cuts other than $U=\mathbb{N}$ or some $U_{K}$ can be used to obtain
new standard results.

\end{document}